\documentclass[12pt]{article}
\usepackage{amssymb,amsthm,amsmath,latexsym}
\usepackage{url}
\newtheorem{thm}{Theorem}[section]
\newtheorem{prop}[thm]{Proposition}
\newtheorem{lem}[thm]{Lemma}

\theoremstyle{remark}
\newtheorem{rem}[thm]{Remark}
\newcommand{\FF}{\mathbb{F}}
\newcommand{\ZZ}{\mathbb{Z}}

\newcommand{\YY}{\mathcal{Y}}
\newcommand{\0}{\mathbf{0}}
\newcommand{\1}{\mathbf{1}}

\DeclareMathOperator{\wt}{wt}
\DeclareMathOperator{\Tdu}{Tdu}

\begin{document}
\title{Binary linear complementary dual codes}

\author{
Masaaki Harada\thanks{
Research Center for Pure and Applied Mathematics,
Graduate School of Information Sciences,
Tohoku University, Sendai 980--8579, Japan.
email: {\tt mharada@m.tohoku.ac.jp}.}
and 
Ken Saito\thanks{
Research Center for Pure and Applied Mathematics,
Graduate School of Information Sciences,
Tohoku University, Sendai 980--8579, Japan.
email: {\tt kensaito@ims.is.tohoku.ac.jp}.}
}
\date{}

\maketitle

\noindent
{\bf Dedicated to Professor Masahiko Miyamoto on His 65th Birthday}

\begin{abstract}
Linear complementary dual codes (or codes with complementary duals)
are codes whose intersections with their dual codes are trivial.
We study binary linear complementary dual $[n,k]$ codes with
the largest minimum weight
among  all binary linear complementary dual $[n,k]$ codes.
We characterize binary linear complementary dual codes 
with the largest minimum weight for small dimensions.
A complete classification of binary linear complementary dual
$[n,k]$ codes with the largest minimum weight 
 is also given for $1 \le k \le n \le 16$.
\end{abstract}

%%%%%%%%%%%%%%%%%%%%%%%%%%%%%%%%%%
\section{Introduction}
An $[n,k]$ {\em code} $C$ over $\FF_q$ 
is a $k$-dimensional vector subspace of $\FF_q^n$,
where $\FF_q$ denotes the finite field of order $q$ and $q$ is a prime power.
A code over $\FF_2$ is called {\em binary}.
The parameters $n$ and $k$
are called the {\em length} and {\em dimension} of $C$, respectively.
The {\em weight} $\wt(x)$ of a vector $x \in \FF_q^n$ is
the number of non-zero components of $x$.
A vector of $C$ is called a {\em codeword} of $C$.
The minimum non-zero weight of all codewords in $C$ is called
the {\em minimum weight} $d(C)$ of $C$ and an $[n,k]$ code with minimum
weight $d$ is called an $[n,k,d]$ code.
%The {\em automorphism group} of $C$ consists of all permutations
%of the coordinates of $C$ which preserve $C$.
Two $[n,k]$ codes $C$ and $C'$ over $\FF_q$ are 
{\em equivalent}, denoted $C \cong C'$,
if there is an $n \times n$ monomial matrix $P$ over $\FF_q$ with 
$C' = C \cdot P = \{ x P \mid x \in C\}$.  
% Two codes are {\em equivalent} if one can be
% obtained from the other by permuting the coordinates.
% An {\em automorphism} of $C$ is a permutation of the coordinates of $C$
% which preserves $C$.
% The set consisting of all automorphisms of $C$ is called the
% {\em automorphism group} of $C$.
% and it is denoted by $\Aut(C)$.

The {\em dual} code $C^{\perp}$ of a code
$C$ of length $n$ is defined as
$
C^{\perp}=
\{x \in \FF_q^n \mid x \cdot y = 0 \text{ for all } y \in C\},
$
where $x \cdot y$ is the standard inner product.
A code $C$ is called {\em linear complementary dual}
(or a linear code with complementary dual)
if $C \cap C^\perp = \{\0_n\}$, where $\0_n$ denotes the zero vector of length $n$.
We say that such a code is LCD for short.

LCD codes were introduced by Massey~\cite{Massey} and gave an optimum linear
coding solution for the two user binary adder channel.
LCD codes are an important class of codes for both
theoretical and practical reasons (see~\cite{CG},
\cite{CMTQ},
\cite{CMTQ2},
\cite{DKOSS},
\cite{bound},
\cite{GOS},
\cite{Jin},
%R \cite{LN},
\cite{Massey},
\cite{MTQ},
\cite{S}).
It is a fundamental problem to classify LCD $[n,k]$ codes
and determine the largest minimum weight
among all LCD $[n,k]$ codes.
Recently, 
much work has been done concerning this fundamental problem
(see~\cite{CMTQ},
\cite{CMTQ2},
\cite{DKOSS},
\cite{bound},
\cite{Jin}).
%R \cite{LN}).
In particular, we emphasize the recent work by 
Carlet, Mesnager, Tang, Qi and Pellikaan~\cite{CMTQ2}.
It has been shown in~\cite{CMTQ2} that
any code over $\FF_q$ is equivalent to some LCD code
for $q \ge 4$.
This motivates us to study binary LCD codes.

Throughout this paper,
let $d(n,k)$ denote the largest minimum weight among all binary
LCD $[n,k]$ codes.
Recently, some bounds on the minimum weights of binary LCD $[n,k]$ codes
have been established in~\cite{bound}.
More precisely, $d(n,2)$ has been determined and 
the values $d(n,k)$ have been calculated for $1 \le k \le n \le 12$.
In this paper, 
% we study binary LCD $[n,k]$ codes 
% with the largest minimum weight among all binary LCD $[n,k]$ codes.
we characterize binary LCD $[n,k,d(n,k)]$ codes for small $k$.
The concept of $k$-covers of $m$-sets
plays an important role in the study of such codes.
Using the characterization, we give a classification of
binary LCD $[n,2,d(n,2)]$ codes and we determine $d(n,3)$.
In this paper, 
% the work is extended for $1 \le k \le n \le 16$.
a complete classification of binary
LCD $[n,k]$ codes having the minimum weight $d(n,k)$
is also given for $1 \le k \le n \le 16$.

The paper is organized as follows.
In Section~\ref{Sec:2},
definitions, notations and basic results are given.
We also give a classification of binary LCD $[n,k,d(n,k)]$ codes 
for $k=1,n-1$.
% In Section~\ref{Sec:R},
% we give recursive constructions of binary LCD codes of small dimensions.
In Section~\ref{Sec:cover},
we give some characterization of binary 
LCD codes using $k$-covers of $m$-sets.
This characterization is used in Sections~\ref{Sec:k2} and \ref{Sec:k3}.
In Section~\ref{Sec:k2}, we study binary LCD codes of dimension $2$.
%We show that
%there are two inequivalent LCD $[6t,2,4t-1]$ codes for $t \ge 1$ and 
%there is a unique LCD code for parameters
%$[6t+1,2,4t]$ ($t \ge 1$),
%$[6t+2,2,4t+1]$ ($t \ge 0$) and $[6t+3,2,4t+2]$ ($t \ge 0$), up to equivalence
We give a classification of binary LCD $[n,2,d(n,2)]$ codes 
for $n=6t$ $(t \ge 1)$, $6t+1$ $(t \ge 1)$,
$6t+2$ $(t \ge 0)$, $6t+3$ $(t \ge 1)$,
$6t+4$ $(t \ge 0)$ and $6t+5$ $(t \ge 0)$
(Theorems~\ref{thm:C1} and \ref{thm:C2}).
In Section~\ref{Sec:k3},
we study binary LCD codes of dimension $3$.
In Section~\ref{Sec:k3}, we show that
$d(n,3) =
\left\lfloor \frac{4n}{7}\right\rfloor$ if $n \equiv 3,5 \pmod {7}$ and
$\left\lfloor \frac{4n}{7}\right\rfloor-1$ otherwise,
for $n \ge 3$ (Theorem~\ref{thm:n3dmax}).
We also establish the uniqueness of binary
LCD $[n,3,d(n,3)]$ codes for 
$n \equiv 0, 2, 3,  5 \pmod {7}$.
Finally, in Section~\ref{Sec:C}, we give a complete classification of
binary LCD $[n,k]$ codes having the minimum weight $d(n,k)$
for $2 \le k \le n-1 \le 15$.

All computer calculations in this paper
were done with the help of {\sc Magma}~\cite{Magma}.

%%%%%%%%%%%%%%%%%%%%%%%%%%%%%%%%
\section{Preliminaries}\label{Sec:2}
\subsection{Definitions, notations and basic results}

% Here, we give definitions, notations and 
% basic results used throughout this paper.

Throughout this paper,
$\0_{s}$ and $\1_{s}$ denote the zero vector and the all-one vector of length $s$, respectively.
Let $I_k$ denote the identity matrix of order $k$ and
let $A^T$ denote the transpose of a matrix $A$.

From now on, all codes mean binary.
Let $C$ be an $[n,k]$ code.  
The {\em weight enumerator} of $C$ is given by
$\sum_{i=0}^n A_i y^i$, where
$A_i$ is the number of codewords of weight $i$ in $C$.
It is trivial that two codes with distinct weight enumerators 
are inequivalent.
The {\em dual} code $C^{\perp}$ of $C$ is defined as
$
C^{\perp}=
\{x \in \FF_2^n \mid x \cdot y = 0 \text{ for all } y \in C\},
$
where $x \cdot y$ is the standard inner product.
A code $C$ is called {\em linear complementary dual}
(or a linear code with complementary dual)
if $C \cap C^\perp = \{\0_n\}$. 
We say that such a code is LCD for short.
% It is easy to see that if $C$ is LCD, then $C^\perp$ is LCD.
A generator matrix of $C$
is a $k \times n$ matrix whose rows are basis vectors of $C$.
A parity-check matrix of $C$ is a generator matrix of $C^\perp$.
The following characterization is due to Massey~\cite{Massey}.

\begin{prop}\label{prop:Massey}
Let $C$ be a code.  
Let $G$ and $H$ be a generator matrix and a parity-check
matrix of $C$, respectively.
Then the following properties are equivalent:
\begin{itemize}
\item[\rm (i)] $C$ is LCD,
\item[\rm (ii)] $C^\perp$ is LCD,
\item[\rm (iii)] $G G^T$ is nonsingular,
\item[\rm (iv)] $H H^T$ is nonsingular.
\end{itemize}
\end{prop}

Throughout this paper,
the condition (iii) is used to verify that a given code 
is LCD.

%R From now on, all codes mean binary unless otherwise specified. 
Let $d(n,k)$ denote the largest minimum weight among all 
LCD $[n,k]$ codes throughout this paper.

%R \begin{lem}\label{lem:2c}
%R Let $G$ (resp.\ $H$) be a generator matrix 
%R (resp.\ a parity-check matrix) of an LCD code.
%R \begin{itemize}
%R \item[\rm (i)]
%R Suppose that some two columns of $G$ are identical.
%R Let $G'$ be the matrix obtained from $G$ by deleting the two
%R columns.
%R Then the code with generator matrix $G'$ is LCD.
%R \item[\rm (ii)]
%R Suppose that some two columns of $H$ are identical.
%R Let $H'$ be the matrix obtained from $H$ by deleting the two
%R columns.
%R Then the code with parity-check matrix $H'$ is LCD.
%R \end{itemize}
%R \end{lem}
%R \begin{proof}
%R Since $G G^T=G'{G'}^T$ and $H H^T=H'{H'}^T$,
%R the new codes are also LCD.
%R \end{proof}

\begin{lem}\label{lem:dual2}
Suppose that there is an LCD $[n,k,d]$ code $C$.
If $d(n-1,k) \le d-1$, then $d(C^\perp) \ge 2$.
\end{lem}
\begin{proof}
Suppose that $d(C^\perp)=1$.
Then some column of a generator matrix of $C$ is $\0_k$.
By deleting the column, an LCD $[n-1,k,d]$ code is constructed.
\end{proof}

\begin{lem}\label{lem:2col}
Suppose that there is an LCD $[n,k,d]$ code $C$ with $d(C^\perp) \ge 2$.
If $n-k \ge 2^k$,
then there is an LCD $[n-2,k]$ code $D$ with $d(D^\perp) \ge 2$.
\end{lem}
\begin{proof}
We may assume without loss of generality that
$C$ has generator matrix of the form
$
G=
\left(
\begin{array}{cccc}
I_k & M 
\end{array}
\right),
$
where $M$ is a $k \times (n-k)$ matrix.
Since $d(C^\perp) \ge 2$, no column of $M$ is $\0_k$.
Since $n-k \ge 2^k$, some two columns of $M$ are identical.
%R By Lemma~\ref{lem:2c}, $D$ is LCD.
Let $G'$ be the matrix obtained from $G$ by deleting the two
columns.
By Proposition~\ref{prop:Massey} (iii),
the code with generator matrix $G'$ is LCD.
\end{proof}

%R Let $C$ be an $[n+1,k,d]$ code with $d(C^\perp)=1$. 
%R Then we may assume without loss of generality that
%R \[
%R C = \{(x_1,x_2,\ldots,x_{n},0) \mid (x_1,x_2,\ldots,x_{n}) \in C^*\},
%R \]
%R where $C^*$ is a punctured $[n,k,d]$ code of $C$.
%R 
%R \begin{lem}\label{lem:p}
%R $C$ is LCD if and only if $C^*$ is LCD. 
%R \end{lem}
%R 
%R In this way, every LCD $[n+1,k,d]$ code $C$ with $d(C^\perp)=1$
%R is constructed from some LCD $[n,k,d]$ code $C^*$.
%R In addition, two LCD $[n+1,k,d]$ codes $C$ with $d(C^\perp)=1$
%R are equivalent if and only if 
%R two LCD $[n,k,d]$ codes $C^*$ are equivalent.
%R Hence, all LCD $[n+1,k,d]$ codes $C$ with $d(C^\perp)=1$,
%R which must be checked to achieve
%R a complete classification, can be obtained from 
%R all inequivalent LCD $[n,k,d]$ codes $C^*$.

The above lemmas are used in 
Sections~\ref{Sec:cover}, \ref{Sec:k2} and \ref{Sec:k3}.

%%%%%%%%%%%%%%%%%%%%%%%%%%%%%%%%%%%%%%%%%%%%%%
\subsection{LCD codes of dimensions $1,n-1$}

It is trivial that $\FF_2^n$ is an LCD $[n,n,1]$ code.
It is known~\cite{DKOSS} that 
\[
(d(n,1),d(n,n-1))=
\begin{cases}
(n,2)   & \text{ if } n \text{ is odd,} \\
(n-1,1) & \text{ if } n \text{ is even.} 
\end{cases}
\]
% Although the following propositions are somewhat trivial,
% we give proofs for the sake of completeness.

The following propositions are trivial,
so we omit the straightforward proofs.

\begin{prop}
There is a unique LCD $[n,1,d(n,1)]$ code, up to equivalence.
\end{prop}
%R \begin{proof}
%R Let $C$ be an LCD $[n,1,d(n,1)]$ code.
%R We may assume without loss of generality that 
%R $C$ has generator matrix of the following form:
%R \[
%R \left(
%R \begin{array}{ccccc}
%R 1&1& \cdots & 1 & 1
%R \end{array}
%R \right) \text{ and }
%R \left(
%R \begin{array}{ccccc}
%R 1&1& \cdots & 1 & 0 
%R \end{array}
%R \right),
%R \]
%R if $n$ is odd and even, respectively.
%R The result follows.
%R \end{proof}

\begin{prop}\label{prop:n-1}
\begin{itemize}
\item[\rm (i)]
Suppose that $n$ is odd.
Then there is a unique LCD $[n,n-1,2]$ code, up to equivalence.
\item[\rm (ii)]
Suppose that $n$ is even.
Then there are $n/2$ inequivalent LCD $[n,n-1,1]$ codes.
\end{itemize}
\end{prop}
%R \begin{proof}
%R Let $C$ be an LCD $[n,n-1,d(n,n-1)]$ code.
%R We may assume without loss of generality that 
%R $C$ has generator matrix of the following form:
%R \[
%R G((a_1,\ldots,a_{n-1}))=
%R \left(
%R \begin{array}{cccc}
%R  &       && a_1 \\
%R  &I_{n-1}&& \vdots\\
%R  &       && a_{n-1}
%R \end{array}
%R \right),
%R \]
%R where $a_i \in \FF_2$ $(i=1,2,\ldots,n-1)$.
%R Then
%R %\begin{equation}\label{eq:H}
%R \[
%R H=
%R \left(
%R \begin{array}{cccccc}
%R a_1 & a_2 & \cdots & a_{n-1} & 1
%R \end{array}
%R \right)
%R %\end{equation}
%R \]
%R is a parity-check matrix of $C$.
%R 
%R \begin{itemize}
%R \item[\rm (i)]
%R Suppose that $n$ is odd.
%R Since $d(C)=2$, $a_i=1$ $(i=1,2,\ldots,n-1)$.
%R Since $n$ is odd,
%R $
%R HH^T= 
%R %\left(
%R %\begin{array}{c}
%R %a_1 + a_2 + \cdots + a_{n-1} + 1
%R %\end{array}
%R %\right)
%R %=
%R \left(
%R \begin{array}{c}
%R 1
%R \end{array}
%R \right)$.
%R Hence, 
%R there is a unique LCD $[n,n-1,2]$ code, up to equivalence.
%R 
%R \item[\rm (ii)]
%R Suppose that $n$ is even.
%R Since  $C$ is LCD,
%R the weight of $(a_1,\ldots,a_{n-1})$ is even.
%R Let $C(x)$ denote the code with generator matrix of the form
%R $G(x)$, where $x \in \FF_2^{n-1}$ and $\wt(x)$ is even.
%R It is easy to see that $C(x)$ and $C(y)$ are equivalent
%R if and only if $\wt(x) = \wt(y)$.
%R Hence, there are $n/2$ LCD $[n,n-1,1]$ codes, up to equivalence.
%R \end{itemize}
%R This completes the proof.
%R \end{proof}

%%%%%%%%%%%%%%%%%%%%%%%%%
\section{Constructions of LCD codes from $k$-covers}
\label{Sec:cover}
In this section, we study LCD codes constructed from $k$-covers of $m$-sets. 
We give a characterization of LCD codes of dimensions $2$ and $3$
using $k$-covers.

%%%%%%%%%%
\subsection{LCD codes from $k$-covers}

Let $m$ and $k$ be positive integers.
Let $X$ be a set with $m$ elements (for short $m$-set).
A {\em $k$-cover} of $X$ is a collection of $k$ not necessarily distinct
subsets of $X$ whose union is $X$~\cite{Clarke}.
This concept plays an important role in the study of LCD codes 
for small dimensions.

We define a generator matrix from a $k$-cover $\{Y_1,Y_2,\ldots,Y_k\}$
of an $m$-set 
$X=\{1,2,\ldots,m\}$ as follow.
Since  the matrix depends on the 
ordering chosen for $Y_1,Y_2,\ldots,Y_k$,
in this paper, we fix the order.
More precisely, we define a $k$-cover as 
a sequence $\YY=(Y_1,Y_2,\ldots,Y_k)$.
Let $\YY=(Y_1,Y_2,\ldots,Y_k)$ be a $k$-cover of $X$.
We define the following subsets of $\{1,2,\ldots,k+\ell m\}$:
\[
\begin{array}{lcl}
Z_1&=&\{1\} \cup (k+Y_1)\cup (k+m+Y_1) \cup \cdots \cup (k+ (\ell-1) m+Y_1),\\
Z_2&=&\{2\} \cup (k+Y_2)\cup (k+m+Y_2) \cup \cdots \cup (k+ (\ell-1) m+Y_2),\\
& \vdots&\\
Z_k&=&\{k\} \cup (k+Y_k)\cup (k+m+Y_k) \cup \cdots \cup (k+ (\ell-1) m+Y_k),
\end{array}
\]
where $\ell$ is an even positive integer and
$a + Y_i =\{a+y \mid y \in Y_i\}$ for a positive integer $a$.
Let $S$ be a subset of $\{1,2,\ldots,s\}$.
We define the binary vector $x=(x_1,x_2,\ldots,x_s)$,
where $x_i=1$ if $i \in S$ and $x_i=0$ otherwise.
This vector $x$ is called the {\em characteristic vector} of $S$.
Let $z_i$ be the characteristic vector of $Z_i$ ($i=1,2,\ldots,k$).
Then define the $k \times (k+ \ell m)$ matrix $G(\YY)$ such that
$z_i$ is the $i$-th row.
% Let $\YY$ be a $k$-cover of an $m$-set $X$.
We denote the code with generator matrix of the form $G(\YY)$
by $C(\YY)$.

\begin{prop}\label{prop:k-cover}
The code $C(\YY)$ is an LCD  $[\ell m +k,k]$ code with $d({C(\YY)}^\perp) = 2$.
\end{prop}
\begin{proof}
Since $\ell$ is even, $G(\YY)G(\YY)^T=I_k$.
Thus, $C(\YY)$ is LCD.
Since $\YY$ is a $k$-cover of $X$,
% $G$ contains no zero column and identical
% columns.
no column of $G(\YY)$ is $\0_k$ and some two columns of $G(\YY)$ are
identical.
This implies that $d(C(\YY)^\perp) = 2$.
\end{proof}

% More generally, we have the following:
% 
% \begin{cor}
% Let $A$ be a $k \times m$ matrix such that no column of $A$ is $\0_k$.
% Let $C$ be a $[k+\ell m,k]$ code with generator matrix of the form
% $
% \left(
% \begin{array}{ccccc}
% I_k & A & \cdots & A
% \end{array}
% \right),
% $
% where $\ell$ is an even positive integer.
% Then $C$ is an LCD  $[k+\ell m,k]$ code with $d(C^\perp) = 2$.
% \end{cor}
% \begin{rem}
% The case for $\ell=2$ can be found 
% in~\cite[Proposition~2]{Massey}.
% \end{rem}

Now we consider the case $k=2,3$ and $\ell=2$.
Let $\YY$ be a $2$-cover and a $3$-cover of $X$, respectively.
Let $C(\YY)$ be a $[2m+2,2]$ code and a $[2m+3,3]$ code
with generator matrices of the form $G(\YY)$, respectively.
Let $C'(\YY)$ denote the $[2m+3,2]$ code and the $[2m+4,3]$ code 
with generator matrices of the following form:
\[
G'(\YY)=
\left(
\begin{array}{ccc}
 & G(\YY) & 
\end{array}
\begin{array}{cc}
1\\
1
\end{array}
\right) \text{ and }
\left(
\begin{array}{ccc}
 & G(\YY) & 
\end{array}
\begin{array}{cc}
0\\
1\\
1
\end{array}
\right),
\]
respectively.

\begin{prop}
The code $C'(\YY)$ is LCD.
\end{prop}
\begin{proof}
For $k=2$ and $3$, the result follows from
\[
G'(\YY)G'(\YY)^T=
\left(
\begin{array}{ccc}
0&1\\
1&0
\end{array}
\right) \text{ and }
\left(
\begin{array}{ccc}
1&0&0\\
0&0&1\\
0&1&0
\end{array}
\right),
\]
respectively.
\end{proof}

The above proposition is used in Propositions~\ref{prop:2cover2} and 
\ref{prop:3cover2}.

%%%%%%%%%%
\subsection{LCD codes from $2$-covers}

In this subsection, we show that
LCD $[n,2]$ codes $C$ with $d(C^\perp) \ge 2$
are constructed from $2$-covers
for $n=2m+2, 2m+3$ $(m \ge 1)$.

\begin{prop}\label{prop:2cover}
Suppose that $m \ge 1$.
Let $C$ be an LCD $[2m+2,2]$ code with $d(C^\perp) \ge 2$.
Then there is a $2$-cover $(Y_1,Y_2)$ of an $m$-set $X$
such that $C \cong C((Y_1,Y_2))$.
\end{prop}
\begin{proof}
We may assume without loss of generality that 
$C$ has generator matrix of the following form:
\begin{equation}\label{eq:M2}
\left(
\begin{array}{cc}
1&0\\
0&1
\end{array}
\begin{array}{cccc}
 & M & 
\end{array}
\right),
\end{equation}
%where $M$ is a $2 \times 2m$ matrix with no zero column.
where $M$ is a $2 \times 2m$ matrix such that no column is $\0_2$.
If $2m \ge 4$, then 
some two columns of $M$ are identical.
Hence, 
an LCD $[2m,2]$ code is constructed
by Lemma~\ref{lem:2col}.
%If $2m \ge 4$, then $M$ contains identical two columns.
%By Lemma~\ref{lem:2c}, an LCD $[2m,2]$ code is constructed.
By continuing this process, 
an LCD $[4,2]$ code with generator
matrix of the form~\eqref{eq:M2} is constructed.
% satisfying that all columns of 
% $M$ are distinct, is constructed, where $n=2,4$.
Hence, we show that such a code is constructed from a $2$-cover.

Since no column of $M$ is $\0_2$, it is sufficient to consider 
the $[4,2]$
codes with generator matrices~\eqref{eq:M2},
where
\[
M=
\left(
\begin{array}{cccc}
0&0\\
1&1
\end{array}
\right), 
\left(
\begin{array}{cccc}
0&1\\
1&0
\end{array}
\right), 
\left(
\begin{array}{cccc}
0&1\\
1&1
\end{array}
\right)
\text{ and }
\left(
\begin{array}{cccc}
1&1\\
1&1
\end{array}
\right).
\]
Only the first code and the last two codes are LCD.
It can be seen by hand that the last two LCD codes are
equivalent.
% In addition, it can be seen by hand that the first code is
% equivalent to the code with generator matrix
% \[
% \left(
% \begin{array}{cccc}
% 1&0&1&1\\
% 0&1&1&1
% \end{array}
% \right).
% \]
This means that the first code and the last code
are 
$C((Y_1,Y_2))$ and $C((Y'_1,Y'_2))$, respectively,
where $Y_1=\emptyset, Y_2=Y'_1=Y'_2=\{1\}$.
%
%By Lemma~\ref{lem:G}, the result follows.
\end{proof}

\begin{prop}\label{prop:2cover2}
Suppose that $m \ge 1$.
Let $C$ be an LCD $[2m+3,2]$ code with $d(C^\perp) \ge 2$.
Then there is a $2$-cover $(Y_1,Y_2)$ of an $m$-set $X$
such that $C \cong C'((Y_1,Y_2))$.
\end{prop}
\begin{proof}
We may assume without loss of generality that 
$C$ has generator matrix of the following form:
\begin{equation}\label{eq:M22}
\left(
\begin{array}{cc}
1&0\\
0&1
\end{array}
\begin{array}{cccc}
 & M' & 
\end{array}
\right),
\end{equation}
%where $M$ is a $2 \times (2m+1)$ matrix with no zero column.
where $M'$ is a $2 \times (2m+1)$ matrix such that no column is $\0_2$.
If $2m+1 \ge 4$, then an LCD $[2m+1,2]$ code is constructed
by Lemma~\ref{lem:2col}.
%If $2m+1 \ge 4$, then $M$ contains identical two columns.
%By Lemma~\ref{lem:2c}, an LCD $[2m+1,2]$ code is constructed.
By continuing this process, 
an LCD $[5,2]$ code with generator
matrix of the form~\eqref{eq:M22} is constructed.

Since no column of $M'$ is $\0_2$, it is sufficient to consider 
the $[5,2]$ codes with generator matrices~\eqref{eq:M22},
where
\begin{multline*}
M'=
\left(
\begin{array}{cccc}
0&0&0\\
1&1&1
\end{array}
\right),
\left(
\begin{array}{cccc}
0&0&1\\
1&1&0
\end{array}
\right),
\left(
\begin{array}{cccc}
0&0&1\\
1&1&1
\end{array}
\right),
\\
\left(
\begin{array}{cccc}
0&1&1\\
1&0&1
\end{array}
\right),
\left(
\begin{array}{cccc}
0&1&1\\
1&1&1
\end{array}
\right)
\text{ and }
\left(
\begin{array}{cccc}
1&1&1\\
1&1&1
\end{array}
\right).
\end{multline*}
Only the third code and the last code are LCD.
It can be seen by hand that the two LCD codes are
equivalent.
In addition, the last code is
$C'((Y_1,Y_2))$, where $Y_1=Y_2=\{1\}$.
This completes the proof.
\end{proof}

% By Proposition~\ref{prop:R2},
% we have the following:
% 
% \begin{cor}%\label{cor:2cover}
% Suppose that $m \ge 0$.
% There is an LCD $[3+2m,2]$ code $C$ with $d(C^\perp) \ge 2$
% if and only if
% there is an LCD $[2+2m,2]$ code $C'$ with $d(C'^\perp) \ge 2$.
% \end{cor}

%%%%%%%%%%
\subsection{LCD codes from $3$-covers}

In this subsection, we show that
LCD $[n,3]$ codes $C$ with $d(C^\perp) \ge 2$
are constructed from $3$-covers
for $n=2m+3, 2m+4$ $(m \ge 1)$.

\begin{prop} \label{prop:3cover}
Suppose that $m \ge 1$.
Let $C$ be an LCD $[2m+3,3]$ code with $d(C^\perp) \ge 2$.
Then there is a $3$-cover $(Y_1,Y_2,Y_3)$ of an $m$-set $X$
such that $C \cong C((Y_1,Y_2,Y_3))$.
\end{prop}
\begin{proof}
We may assume without loss of generality that 
$C$ has generator matrix of the following form:
\begin{equation}\label{eq:M}
\left(
\begin{array}{cccccccc}
1&0&0&&   & \\
0&1&0&& M & \\
0&0&1&&   &
\end{array}
\right),
\end{equation}
%where $M$ is a $3 \times 2m$ matrix with no zero column.
where $M$ is a $3 \times 2m$ matrix such that no column is $\0_3$.
If $2m \ge 8$, then an LCD $[2m+1,3]$ code is constructed
by Lemma~\ref{lem:2col}.
%If $2m \ge 8$, then $M$ contains identical two columns.
%By Lemma~\ref{lem:2c}, an LCD $[1+2m,3]$ code is constructed.
By continuing this process, 
an LCD $[n,3]$ code with generator
matrix of the form~\eqref{eq:M} 
%satisfying that all columns of $M$ are distinct, 
is constructed, where $n=5,7,9$.
Hence, we show that such a code is constructed from a $3$-cover.

Let $C_9$ be an LCD $[9,3]$ code with generator
matrix of the form~\eqref{eq:M} satisfying that all columns of 
$M$ are distinct.
Our computer search shows that $C_9$ is equivalent to 
%R the code $D_9$ with generator matrix
%R \[
%R \left(
%R \begin{array}{ccccccccc}
%R 1& 0& 0& 1& 0& 1& 0& 1& 1\\
%R 0& 1& 0& 0& 1& 1& 0& 0& 1\\
%R 0& 0& 1& 0& 0& 0& 1& 1& 1
%R \end{array}
%R \right).
%R \]
%R In addition, our computer search shows that $D_9$
%R is equivalent to 
the code with generator matrix
\[
\left(
\begin{array}{ccccccccc}
1& 0& 0& 1& 1& 1& 1& 1& 1\\
0& 1& 0& 1& 0& 1& 1& 0& 1\\
0& 0& 1& 1& 1& 0& 1& 1& 0\\
\end{array}
\right).
\]
This means that the code is $C((Y_1,Y_2,Y_3))$, where
$Y_1=\{1,2,3\}$, $Y_2=\{1,3\}$ and $Y_3=\{1,2\}$.

Let $C_7$ be an LCD $[7,3]$ code with generator
matrix of the form~\eqref{eq:M} satisfying that all columns of 
$M$ are distinct.
Our computer search shows that $C_7$ is equivalent to one of 
%R the codes $D_{7,1}$ and $D_{7,2}$ with generator matrices
%R \[
%R \left(
%R \begin{array}{ccccccccc}
%R 1& 0& 0& 1& 0& 1& 1\\
%R 0& 1& 0& 0& 1& 1& 0\\
%R 0& 0& 1& 0& 0& 0& 1
%R \end{array}
%R \right)
%R \text{ and }
%R \left(
%R \begin{array}{ccccccccc}
%R 1& 0& 0& 1& 0& 1& 0\\
%R 0& 1& 0& 0& 1& 0& 1\\
%R 0& 0& 1& 0& 0& 1& 1
%R \end{array}
%R \right),
%R \]
%R respectively.
%R In addition, our computer search shows that $D_{7,1}$ and $D_{7,2}$
%R are equivalent to 
the codes with generator matrices
\[
\left(
\begin{array}{ccccccccc}
1& 0& 0& 1& 1& 1& 1\\
0& 1& 0& 1& 1& 1& 1\\
0& 0& 1& 1& 0& 1& 0
\end{array}
\right)
\text{ and }
\left(
\begin{array}{ccccccccc}
1& 0& 0& 1& 1& 1& 1\\
0& 1& 0& 0& 1& 0& 1\\
0& 0& 1& 1& 0& 1& 0
\end{array}
\right).
\]
This means that the codes are 
$C((Y_1,Y_2,Y_3))$ and $C((Y'_1,Y'_2,Y'_3))$, respectively, where
$Y_1=Y_2=Y'_1=\{1,2\}$, $Y_3=Y'_3=\{1\}$ and $Y'_2=\{2\}$.

Our computer search shows that an LCD $[5,3]$ code 
is equivalent to one of 
the codes with generator matrices
\[
\left(
\begin{array}{ccccccccc}
1& 0& 0& 1& 1\\
0& 1& 0& 0& 0\\
0& 0& 1& 0& 0
\end{array}
\right), 
\left(
\begin{array}{ccccccccc}
1& 0& 0& 1& 1\\
0& 1& 0& 1& 1\\
0& 0& 1& 0& 0
\end{array}
\right)\text{ and }
\left(
\begin{array}{ccccccccc}
1& 0& 0& 1& 1\\
0& 1& 0& 1& 1\\
0& 0& 1& 1& 1
\end{array}
\right).
\]
This means that the codes are 
$C((Y_1,Y_2,Y_3))$, $C((Y'_1,Y'_2,Y'_3))$ 
and $C((Y''_1,Y''_2,Y''_3))$, respectively,
where
$Y_1=Y'_1=Y'_2=Y''_1=Y''_2=Y''_3=\{1\}$ and 
$Y_2=Y_3=Y'_3=\emptyset$.
%
%By Lemma~\ref{lem:G}, the result follows.
\end{proof}

% \begin{Q}
% Let ${\mathcal Y}_1=\{Y_1,Y_2,Y_3\}$ and 
% ${\mathcal Y}_2=\{Y'_1,Y'_2,Y'_3\}$ be
% (distinct) disordered $3$-covers of an unlabelled $m$-set.
% Are
% $C({\mathcal Y}_1)$ and $C({\mathcal Y}_2)$
% inequivalent?
% \end{Q}

\begin{prop}\label{prop:3cover2}
Suppose that $m \ge 1$.
Let $C$ be an LCD $[2m+4,3]$ code with $d(C^\perp) \ge 2$.
Then there is a $3$-cover $(Y_1,Y_2,Y_3)$ of an $m$-set $X$
such that $C \cong C'((Y_1,Y_2,Y_3))$.
\end{prop}
\begin{proof}
We may assume without loss of generality that 
$C$ has generator matrix of the following form:
\begin{equation}\label{eq:M3cover2}
\left(
\begin{array}{cccccccc}
1&0&0&&   & \\
0&1&0&& M' & \\
0&0&1&&   &
\end{array}
\right),
\end{equation}
%where $M$ is a $3 \times (2m+1)$ matrix with no zero column.
where $M'$ is a $3 \times (2m+1)$ matrix such that no column is $\0_3$.
If $2m+1 \ge 8$, then an LCD $[2m+2,3]$ code is constructed
by Lemma~\ref{lem:2col}.
%If $(2m+1) \ge 8$, then $M$ contains identical two columns.
%By Lemma~\ref{lem:2c}, an LCD $[2+2m,3]$ code is constructed.
By continuing this process, 
an LCD $[n,3]$ code with generator
matrix of the form~\eqref{eq:M3cover2} 
%satisfying that all columns of $M$ are distinct, 
is constructed, where $n=6,8,10$.

Let $C_{10}$ be an LCD $[10,3]$ code with generator
matrix of the form~\eqref{eq:M3cover2} satisfying that all columns of 
$M'$ are distinct.
%R Then $C_{10}$ is equivalent to 
%R the code $D_{10}$ with generator matrix
%R \[
%R \left(
%R \begin{array}{cccccccccc}
%R 1&0&0&1&0&1&0&1&0&1\\
%R 0&1&0&0&1&1&0&0&1&1\\
%R 0&0&1&0&0&0&1&1&1&1
%R %1& 0& 0& 1& 1& 0& 0& 0& 1 &1 \\
%R %0& 1& 0& 1& 0& 1& 0& 1& 0 &1 \\
%R %0& 0& 1& 1& 0& 0& 1& 1& 1 &0
%R \end{array}
%R \right).
%R \]
%R Our computer search shows that $D_{10}$ is equivalent to the code
Our computer search shows that $C_{10}$ is equivalent to the code
with generator matrix
\[
\left(
\begin{array}{cccccccccc}
1& 0& 0& 1& 1& 1& 1& 1& 1&0\\
0& 1& 0& 1& 0& 1& 1& 0& 1&1\\
0& 0& 1& 1& 1& 0& 1& 1& 0&1\\
\end{array}
\right).
\]
This means that the code is $C'((Y_1,Y_2,Y_3))$, where
$Y_1=\{1,2,3\}$, $Y_2=\{1,3\}$ and $Y_3=\{1,2\}$.
%Let $G_{9}$ be the matrix obtained by removing the last column of $G_{10}$.
%One can calculate that $G_{9}G_{9}^T$ is nonsingular.
%Hence, the code with generator matrix $G_{9}$ is the LCD $[9,3]$ code.

Let $C_8$ be an LCD $[8,3]$ code with generator
matrix of the form~\eqref{eq:M3cover2} satisfying that all columns of 
$M'$ are distinct.
Our computer search shows that $C_8$ is equivalent to 
%R the code $D_8$ with generator matrix
%R \[
%R \left(
%R \begin{array}{ccccccccc}
%R 1&0&0&1&0&1&1&0\\
%R 0&1&0&0&1&1&0&1\\
%R 0&0&1&0&0&0&1&1\\
%R %1& 0& 0& 1& 0& 0& 1 &1 \\
%R %0& 1& 0& 0& 1& 1& 0 &1 \\
%R %0& 0& 1& 0& 0& 1& 1 &0
%R \end{array}
%R \right).
%R \]
%R In addition, our computer search shows that $D_{8}$ is equivalent to 
the code with generator matrix
\[
\left(
\begin{array}{cccccccc}
1& 0& 0& 1& 1& 1& 1&0\\
0& 1& 0& 1& 1& 1& 1&1\\
0& 0& 1& 1& 0& 1& 0&1
\end{array}
\right).
\]
This means that the code is $C'((Y_1,Y_2,Y_3))$, where
$Y_1=Y_2=\{1,2\}$ and $Y_3=\{1\}$.
%Let $G_{7}$ be the matrix obtained by removing the last column of $G_{8}$.
%One can calculate that $G_{7}G_{7}^T$ is nonsingular.
%Hence, the code with generator matrix $G_{7}$ is the LCD $[7,3]$ code.

Our computer search shows that an LCD $[6,3]$ code
is equivalent to one of 
the codes with generator matrices
$
\left(
\begin{array}{cc}
I_3 & A
\end{array}
\right),
$
where
\[
A=
\left(
\begin{array}{ccccccccc}
0&0&0\\
1&1&1\\
1&1&1
\end{array}
\right), 
\left(
\begin{array}{ccccccccc}
1&1&0\\
0&0&1\\
0&0&1
\end{array}
\right) \text{ and }
\left(
\begin{array}{ccccccccc}
1&1&0\\
1&1&1\\
1&1&1
\end{array}
\right).
\]
This means that the codes are 
$C'((Y_1,Y_2,Y_3))$, 
$C'((Y'_1,Y'_2,Y'_3))$ and 
$C'((Y''_1,Y''_2,Y''_3))$, 
respectively,
where
$Y_1=Y'_2=Y'_3=\emptyset$ and
$Y_2=Y_3=Y'_1=Y''_1=Y''_2=Y''_3=\{1\}$.
%Let $G_{5}$ be the matrix obtained by removing the last column of $G_{6}$.
%One can calculate that $G_{5}G_{5}^T$ is nonsingular.
%Hence, the code with generator matrix $G_{5}$ is the LCD $[5,3]$ code.
%
%Hence, by Proposition~\ref{prop:3cover}, $C$ is equivalent to a code having the
%following generator matrix
%\begin{equation}\label{eq:M}
%G'=
%\left(
%\begin{array}{ccccccc}
%1&0&0&   &&  \\
%0&1&0& M' & M' & x^T \\
%0&0&1&   && 
%\end{array}
%\right),
%\end{equation}
%where $M'$ is a $3 \times m$ matrix with no zero column and $x \in \FF_2^3$ is
%a nonzero vector.
%Since $G'G'^T$ is nonsingular, the weight of $x$ must be $2$.
%This completes the proof.
\end{proof}

% By Proposition~\ref{prop:R3},
% we have the following:
% 
% \begin{cor} %\label{cor:3cover}
% Suppose that $m \ge 0$.
% There is an LCD $[4+2m,3]$ code $C$ with $d(C^\perp) \ge 2$
% if and only if 
% there is an LCD $[3+2m,3]$ code $C'$ with $d(C'^\perp) \ge 2$.
% \end{cor}

%%%%%%%%%%
%\subsection{LCD codes from $4$-covers}
\subsection{Remarks}

The elements of an $m$-set $X$ may be taken to be identical.
In this case, $X$ is called {\em unlabelled}.
Let $\YY=(Y_1,Y_2,\ldots,Y_k)$ be a $k$-cover of $X$.
The order of the sets $Y_1,Y_2,\ldots,Y_k$ may not be material.
In this case, $\YY$ is called {\em disordered}~\cite{Clarke}.

\begin{prop}\label{prop:disordered}
Let $\YY$ be a $k$-cover of an $m$-set $X$.
Let $\YY'$ be the $k$-cover obtained from $\YY$ by
a permutation of $Y_1,Y_2,\ldots,Y_k$ and
a permutation of the elements of $X$.
Then $C(\YY)  \cong  C(\YY')$.
\end{prop}
\begin{proof}
Consider the generator matrix $G(\YY)$ of $C(\YY)$
constructed from a $k$-cover $\YY=(Y_1,Y_2,\ldots,Y_k)$.
A permutation of $Y_1,Y_2,\ldots,Y_k$ implies 
a permutation of rows of $G(\YY)$.
A permutation of the elements of $X$ implies 
a permutation of columns of $G(\YY)$.
The result follows.
\end{proof}

By the above proposition, 
when we consider codes $C(\YY)$ constructed from all $k$-covers $\YY$,
which must be checked to achieve a complete classification, 
it is sufficient to consider
only disordered $k$-covers of unlabelled $m$-sets.

Now let us consider LCD codes constructed from $4$-covers.
Our computer search shows that there are six
inequivalent LCD $[6,4]$ codes $D_{6,i}$ $(i=1,2,\ldots,6)$
with $d(D_{6,i}^\perp) \ge 2$.
These codes $D_{6,i}$ have generator matrices
$
\left(
\begin{array}{cc}
I_4 & A
\end{array}
\right),
$
where
\begin{multline*}
A=
\left(
\begin{array}{cc}
1& 1\\
0& 0\\
0& 0\\
0& 0
\end{array}
\right),
\left(
\begin{array}{cc}
1& 1\\
1& 1\\
0& 0\\
0& 0
\end{array}
\right),
\left(
\begin{array}{cc}
1& 1\\
1& 1\\
1& 1\\
0& 0
\end{array}
\right),
\\
\left(
\begin{array}{cc}
1& 1\\
1& 1\\
1& 1\\
1& 1
\end{array}
\right),
\left(
\begin{array}{cc}
1& 0\\
1& 0\\
0& 1\\
0& 1
\end{array}
\right)
\text{ and }
\left(
\begin{array}{cc}
1& 0\\
1& 0\\
0& 1\\
1& 1
\end{array}
\right),
\end{multline*}
respectively.
The weight enumerators $W_{6,i}$ of the codes $D_{6,i}$ are
listed in Table~\ref{Tab:W4}.
% The codes $C_i$ have the following weight enumerators:
% \[
% \begin{array}{ll}
% 1+ 3y + 3y^2 + 2y^3 + 3y^4 + 3y^5+ y^6, &
% 1+ 2y + 2y^2 + 4y^3 + 5y^4 + 2y^5, \\
% 1+ y + 3y^2 + 6y^3 + 3y^4 + y^5 + y^6, &
% 1+6y^2 + 4y^3 + y^4 + 4y^5, \\
% 1+6y^2 + 9y^4, &
% 1+4y^2 +6y^3 +3y^4 +2y^5, \\
% \end{array}
% \]
% respectively.
It is easy to see that the number of 
disordered $4$-covers of an unlabelled $1$-set
is $4$~\cite[Table~1]{Clarke}.
Only the codes $D_{6,i}$ $(i=1,2,3,4)$ are constructed from $4$-covers.

%%%%%%%%%%%%%%%%%%%%%%%%%%%%%%
\begin{table}[thb]
\caption{$W_{6,i}$ $(i=1,2,\ldots,6)$}
\label{Tab:W4}
\begin{center}
%{\small
{\footnotesize
%{\scriptsize
%{\tiny
\begin{tabular}{c|l|c|l}
\noalign{\hrule height0.8pt}
$i$ & \multicolumn{1}{c|}{$W_{6,i}$}&$i$ & \multicolumn{1}{c}{$W_{6,i}$}\\
\hline
1& $1+ 3y + 3y^2 + 2y^3 + 3y^4 + 3y^5+ y^6$&
4& $1+6y^2 + 4y^3 + y^4 + 4y^5$ \\
2& $1+ 2y + 2y^2 + 4y^3 + 5y^4 + 2y^5$ &
5& $1+6y^2 + 9y^4$ \\
3& $1+ y + 3y^2 + 6y^3 + 3y^4 + y^5 + y^6$ &
6& $1+4y^2 +6y^3 +3y^4 +2y^5$ \\
\noalign{\hrule height0.8pt}
\end{tabular}
}
\end{center}
\end{table}
%%%%%%%%%%%%%%%%%%%%%%%%%%%%%%%

% According to~\cite{Clarke},
% let $\Tdu(m,k)$ denote
% the number of disordered $k$-covers of an unlabelled $m$-set.
% 
% \begin{cor}
% $N(2m+3,2m,2) \le \Tdu(m,3)$.
% \end{cor}
% 
% The formula $\Tdu(m,k)$ is given in~\cite[Theorem 2]{Clarke}.
% For $m \le 7$ and $k \le 8$, $\Tdu(m,k)$
% is numerically determined in~\cite[Table 1]{Clarke}.

%%%%%%%%%%%%%%%%%%%%%%%%%%%%%%%%%%%%%%%%%%%%%%
\section{LCD codes of dimension 2}
\label{Sec:k2}

It was shown in~\cite{bound} that
\[
d(n,2) =
\begin{cases}
\lfloor \frac{2n}{3} \rfloor  & \text{ if } n \equiv 1,2,3,4 \pmod 6,\\
\lfloor \frac{2n}{3} \rfloor-1& \text{ otherwise},
\end{cases}
\]
for $n \ge 2$.
Throughout this section, we denote $d(n,2)$ by $d_n$.
In this section,
we give a classification of LCD $[n,2,d_n]$ codes 
for $n=6t$ $(t \ge 1)$, $6t+1$ $(t \ge 1)$,
$6t+2$ $(t \ge 0)$, $6t+3$ $(t \ge 1)$,
$6t+4$ $(t \ge 0)$ and $6t+5$ $(t \ge 0)$.
In Section~\ref{Sec:cover}, we gave some observation of LCD codes of
dimension $2$, which is established from the concept of 
$2$-covers of $m$-sets.
The observation is useful to complete the classification.

\begin{lem}\label{lem:2-0}
Suppose that $n \ge 2$ and $n \equiv 0,1,2,3 \pmod{6}$.
If there is an LCD $[n,2,d_n]$ code $C$, then $d(C^\perp) \ge 2$.
\end{lem}
\begin{proof}
Write $n=6t+s$, where $0 \le s \le 5$.
For $s$ and $d_n$, we have the following:
\begin{center}
\begin{tabular}{cc|cc|cc}
$s$ & $d_n$& $s$ & $d_n$& $s$ & $d_n$\\
\hline
0& $4t-1$ &2 & $4t+1$ &4 &$4t+2$\\
1& $4t$    &3 & $4t+2$  &5 &$4t+2$
\end{tabular}
\end{center}
The result follows by Lemma~\ref{lem:dual2}.
\end{proof}

Now suppose that $C$ and $C'$ are an LCD $[2m+2,2]$ code 
with $d(C^\perp) \ge 2$
and an LCD $[2m+3,2]$ code
with $d(C'^\perp) \ge 2$,
respectively, for $m \ge 1$.
By Propositions~\ref{prop:2cover} and~\ref{prop:2cover2},
we may assume without loss of generality that
$C$ and $C'$ have generator matrices of the following form:
\begin{align*}
G^0(a,b,c)=&
\left(
\begin{array}{cc}
1&0\\
0&1
\end{array}
\begin{array}{cccc}
 M(a,b,c)& M(a,b,c)
\end{array}
\right) \text{ and }
\\
G^1(a,b,c)=&
\left(
\begin{array}{cc}
1&0\\
0&1
\end{array}
\begin{array}{cccc}
 M(a,b,c)& M(a,b,c)
\end{array}
\begin{array}{cc}
1\\
1
\end{array}
\right),
\end{align*}
respectively,
where
\begin{equation}\label{eq:Mabc}
M(a,b,c)
=
\left(
\begin{array}{ccccccc}
\1_{a} & \1_{b} & \0_{c} \\
\1_{a} & \0_{b} & \1_{c} \\
\end{array}
\right).
\end{equation}
We denote the codes with generator matrices 
$G^0(a,b,c)$ and  $G^1(a,b,c)$
by
$C^0(a,b,c)$ and $C^1(a,b,c)$, respectively.
Then the codes $C^\delta(a,b,c)$ have the following weight enumerators for $\delta \in \{0,1\}$:
\begin{align}\label{eq:2we}
1
+y^{1+2(a+b)+\delta}+y^{1+2(a+c)+\delta}
+y^{2+2(b+c)}.
\end{align}

For nonnegative integers $a,b,c,n$ and $\delta \in \{0,1\}$, we consider the following
conditions:
\begin{align}
&d_n \le 1+2(a+b)+\delta,   \label{eq:ab}\\
&d_n \le 1+2(a+c)+\delta,   \label{eq:ac}\\
&d_n \le 2+2(b+c),   \label{eq:bc}\\
&2(a+b+c)+2 +\delta =n,     \label{eq:abc} \\
&b \le c. \label{eq:blec}
\end{align}
We note that the conditions \eqref{eq:ab}--\eqref{eq:bc}
are related to the minimum
weight of $C^\delta(a,b,c)$.

\begin{lem}\label{lem:abc}
\begin{itemize}
\item[\rm (i)]
Let $S$ be the set of $(a,b,c)$ satisfying the 
conditions~\eqref{eq:ab}--\eqref{eq:blec}, where $\delta=1$.
\begin{itemize}
\item[\rm (1)]
If $n=6t+1$ $(t \ge 1)$, then $S=\{(t-1,t,t),(t,t-1,t)\}$.
\item[\rm (2)]
If $n=6t+3$ $(t \ge 1)$, then $S=\{(t,t,t)\}$.
\item[\rm (3)]
If $n=6t+5$ $(t \ge 1)$, then
\[
S=
\left\{
\begin{array}{l}
(t-1,t+1,t+1),
(t,t,t+1),\\
(t+1,t-1,t+1),
(t+1,t,t)
\end{array}
\right\}.
\]
\end{itemize}
\item[\rm (ii)]
Let $S$ be the set of $(a,b,c)$ satisfying the 
conditions~\eqref{eq:ab}--\eqref{eq:blec}, where $\delta=0$.
\begin{itemize}
\item[\rm (1)]
If $n=6t$ $(t \ge 1)$, then $S=\{(t-1,t,t),(t,t-1,t)\}$.
\item[\rm (2)]
If $n=6t+2$ $(t \ge 1)$, then $S=\{(t,t,t)\}$.
\item[\rm (3)]
If $n=6t+4$ $(t \ge 0)$, then $S=\{(t+1,t,t)\}$.
\end{itemize}
\end{itemize}
\end{lem}
\begin{proof}
All cases are similar, and we only give the details
for $n=6t+1$.

From~\eqref{eq:bc} and \eqref{eq:abc}, we have
$
a \le t.
$
From~\eqref{eq:ab}, \eqref{eq:ac} and \eqref{eq:abc}, we have
$
t-1 \le a.
$
Thus, we have
\[
a \in \{t-1,t\}.
\]

%%%%%%%%%%%%
Suppose that $a=t-1$.
From~\eqref{eq:ab}, we have
$t \le b$. 
From~\eqref{eq:ac}, we have
$t \le c$. 
From~\eqref{eq:abc}, we have
$b+c=2t$.
Hence, we have
$b=c=t$. 

%%%%%%%%%%%%
Suppose that $a=t$.
From~\eqref{eq:ab}, we have
$t-1 \le b$. 
From~\eqref{eq:ac}, we have
$t-1 \le c$. 
From~\eqref{eq:abc}, we have
$b+c=2t-1$.
From~\eqref{eq:blec}, we have
$(b,c)=(t-1,t)$.
\end{proof}

\begin{lem}\label{lem:bc}
$C^{\delta}(a,b,c)  \cong C^{\delta}(a,c,b)$ for $\delta \in \{0,1\}$.
\end{lem}
\begin{proof}
The matrix $G^{\delta}(a,c,b)$ is obtained from $G^{\delta}(a,c,b)$ by
permutations of rows and columns.
\end{proof}

\begin{lem}\label{lem:equiv}
$C^1(a,b,c)  \cong C^1(b,a,c) \cong C^1(c,b,a)$.
\end{lem}
\begin{proof}
We denote the code with generator matrix of the form
$M(a,b,c)$ in~\eqref{eq:Mabc} by $D(a,b,c)$.
Let $r_i$ be the $i$-th row of $M(a,b,c)$.
By considering the matrices
${\displaystyle
\left(
\begin{array}{ccccccc}
r_1\\
r_1+r_2
\end{array}
\right)
}$ and
${\displaystyle
\left(
\begin{array}{ccccccc}
r_1+r_2\\
r_2\\
\end{array}
\right)
}$,
we have $D(a,b,c) =D(b,a,c)=D(c,b,a)$. 
Since $C^1(a,b,c)  \cong D(2a+1,2b+1,2c+1)$,
the result follows.
\end{proof}

The above two lemmas are used for 
a classification of LCD $[n,2,d_n]$ codes.

\begin{thm}\label{thm:C1}
\begin{itemize}
\item[\rm (i)]
For $t \ge 1$,
there are two inequivalent LCD $[6t,2,4t-1]$ codes.
\item[\rm (ii)]
For $t \ge 1$,
there is a unique LCD $[6t+1,2,4t]$ code, up to equivalence.
\item[\rm (iii)]
For $t \ge 1$,
there is a unique LCD $[6t+2,2,4t+1]$ code, up to equivalence.
\item[\rm (iv)]
For $t \ge 1$,
there is a unique LCD $[6t+3,2,4t+2]$ code, up to equivalence.
\end{itemize}
\end{thm}
\begin{proof}
Let $C$ be an LCD $[n,2]$ code for $n \ge 4$.
For the parameters $[6t,2,4t-1]$,
$[6t+1,2,4t]$, $[6t+2,2,4t+1]$ and $[6t+3,2,4t+2]$ ($t \ge 1$),
by Lemma~\ref{lem:2-0},
we may assume without loss of generality that 
$C$ has generator matrix of the form $G^\delta(a,b,c)$ for
$\delta=0,1,0,1$, respectively.
In addition, $C$ satisfies~\eqref{eq:ab}--\eqref{eq:abc}.
By Lemma~\ref{lem:bc},
we may assume without loss of generality that 
$C$ satisfies~\eqref{eq:blec}.

\begin{itemize}
\item[\rm (i)]
Assume that $n=6t$ $(t \ge 1)$.
By Lemma~\ref{lem:abc}~(ii), 
$(a,b,c)$ is $(t-1,t,t)$ or $(t,t-1,t)$.
Let $C_1$ and $C_2$ be the LCD codes with generator matrices
$G^0(a,b,c)$ for these $(a,b,c)$, respectively.
By~\eqref{eq:2we},
the codes $C_1$ and $C_2$ have 
the following weight enumerators:
\begin{align*}
1 + 2y^{4t-1} + y^{4t+2} \text{ and } 
1 + y^{4t-1} + y^{4t} + y^{4t+1}, 
\end{align*}
respectively.
Hence, the two codes are inequivalent.

\item[\rm (ii)]
Assume that $n=6t+1$ ($t \ge 1$).
By Lemma~\ref{lem:abc}~(i), 
$(a,b,c)$ is $(t-1,t,t)$ or $(t,t-1,t)$.
Let $C_1$ and $C_2$ be the LCD codes with generator matrices
$G^1(a,b,c)$ for these $(a,b,c)$, respectively.
By Lemma~\ref{lem:equiv}, 
$C_1$ and $C_2$ are equivalent.

\item[\rm (iii)]
For $n=6t+2$ $(t \ge 1)$,
the uniqueness follows from Lemma~\ref{lem:abc}~(ii).

\item[\rm (iv)]
For $n=6t+3$ $(t \ge 1)$,
the uniqueness follows from Lemma~\ref{lem:abc}~(i).
\end{itemize}
This completes the proof.
\end{proof}

We remark that
there is a unique LCD $[3,2,2]$ code, up to equivalence,
by Proposition~\ref{prop:n-1}.

\begin{lem}\label{lem:C2}
\begin{itemize}
\item[\rm (i)]
For $t \ge 0$,
there is a unique LCD $[6t+4,2,4t+2]$ code $C$ with
$d(C^\perp) \ge 2$, up to equivalence.
\item[\rm (ii)]
For $t \ge 1$,
there are two inequivalent LCD $[6t+5,2,4t+2]$ codes $C$ with
$d(C^\perp) \ge 2$.
\end{itemize}
\end{lem}
\begin{proof}
Let $C$ be an LCD $[n,2]$ code with $d(C^\perp) \ge 2$ and $n \ge 4$.
For the parameters 
$[6t+4,2,4t+2]$ ($t \ge 0$) and
$[6t+5,2,4t+2]$ ($t \ge 1$),
since $d(C^\perp)\ge 2$, 
we may assume without loss of generality that 
$C$ has generator matrix of the form $G^\delta(a,b,c)$ for
$\delta=0,1$, respectively.
In addition, $C$ satisfies~\eqref{eq:ab}--\eqref{eq:abc}.
By Lemma~\ref{lem:bc},
we may assume without loss of generality that 
$C$ satisfies~\eqref{eq:blec}.

\begin{itemize}
\item[\rm (i)]
For $n=6t+4$ $(t \ge 0)$,
the uniqueness follows from Lemma~\ref{lem:abc}~(ii).

\item[\rm (ii)]
Assume that $n=6t+5$ $(t \ge 1)$.
By Lemma~\ref{lem:abc}~(i), 
$(a,b,c)$ is $(t-1,t+1,t+1)$, $(t,t,t+1)$, $(t+1,t-1,t+1)$ or $(t+1,t,t)$.
Let $C_i$ $(i=1,2,3,4)$ be the LCD codes with generator 
matrices
$G^1(a,b,c)$ for these $(a,b,c)$, respectively.
By Lemma~\ref{lem:equiv}, 
$C_1 \cong C_3$ and $C_2 \cong C_4$.
By~\eqref{eq:2we},
the codes $C_1$ and $C_2$ have the following
weight enumerators:
\begin{align*}
%R 1 + 2y^{4t+2} + y^{4t+4} \text{ and }
%R 1 + y^{4t+2} + 2y^{4t+4}, 
1 + 2y^{4t+2} + y^{4t+6} \text{ and }
1 + y^{4t+2} + 2y^{4t+4},
\end{align*}
respectively.
Hence, the two codes are inequivalent.
\end{itemize}
This completes the proof.
\end{proof}

\begin{rem}
By~\cite[Theorem~3]{bound}, 
the dual codes of the codes given in the above lemma
have minimum weight $2$. 
\end{rem}

\begin{thm}\label{thm:C2}
\begin{itemize}
\item[\rm (i)]
For $t \ge 0$,
there are two inequivalent LCD $[6t+4,2,4t+2]$ codes.
\item[\rm (ii)]
For $t \ge 1$,
there are four inequivalent LCD $[6t+5,2,4t+2]$ codes.
\end{itemize}
\end{thm}
\begin{proof}
%R Let $C$ be an $[n+1,k,d]$ code with $d(C^\perp)=1$. 
%R Then we may assume without loss of generality that
%R \[
%R C = \{(x_1,x_2,\ldots,x_{n},0) \mid (x_1,x_2,\ldots,x_{n}) \in C^*\},
%R \]
%R where $C^*$ is a punctured $[n,k,d]$ code of $C$.
%R 
%R \begin{lem}\label{lem:p}
%R $C$ is LCD if and only if $C^*$ is LCD. 
%R \end{lem}
%R By Lemma~\ref{lem:p}, 
It is easy to see that
all LCD $[n+1,k,d]$ codes $C$ with $d(C^\perp)=1$,
which must be checked to achieve
a complete classification, can be obtained from
all inequivalent LCD $[n,k,d]$ codes.

\begin{itemize}
\item[\rm (i)]
By Theorem~\ref{thm:C1},
there is a unique LCD $[6t+3,2,4t+2]$ code, up to equivalence,
for $t \ge 1$.
The result follows from Lemma~\ref{lem:C2}.
\item[\rm (ii)] 
The result follows from Lemma~\ref{lem:C2} and the part (i).
\end{itemize}
This completes the proof.
\end{proof}

We remark that
there are three inequivalent LCD $[5,2,2]$ codes
(see Table~\ref{Tab:C}).

%%%%%%%%%%%%%%%%%%%%%%%%%%%%%%%%%%%%%%%%%%%%%%
\section{LCD codes of dimension 3}
\label{Sec:k3}

The aim of this section is to establish the following theorem.
In Section~\ref{Sec:cover}, we gave some observation of LCD codes of
dimension $3$, which is established from the concept of 
$3$-covers of $m$-sets.
The observation is useful to do this.

\begin{thm}\label{thm:n3dmax}
For $n \ge 3$,
\[
d(n,3) =
\begin{cases}
\left\lfloor \frac{4n}{7}\right\rfloor & \text{ if } n \equiv 3,5 \pmod {7},\\
\left\lfloor \frac{4n}{7}\right\rfloor-1& \text{ otherwise.}
\end{cases}
\]
\end{thm}

In this section, 
we also establish the uniqueness of 
LCD $[n,3,d(n,3)]$ codes for 
$n \equiv 0, 2, 3,  5 \pmod {7}$ and $n \ge 5$.

Throughout this section, 
we denote $\left\lfloor \frac{4n}{7}\right\rfloor$ by $\alpha_n$.

\begin{lem}\label{lem:3-00}
There is no LCD $[n,3,\alpha_n]$ code
for $n \equiv 2 \pmod {7}$.
\end{lem}
\begin{proof}
% By~\cite[Lemma 2]{bound}, it holds that
% $d(n,3) \le \left\lfloor \frac{4n}{7}\right\rfloor$
% for $n \ge 3$.
Suppose that there is an (unrestricted) $[n,3,d]$ code.
By the Griesmer bound, we have
\[
n \ge d+ \left\lceil \frac{d}{2} \right\rceil 
+ \left\lceil \frac{d}{4} \right\rceil.
\]
Hence, we have
\[
d(n,3) \le
\begin{cases}
%\left\lfloor \frac{4n}{7}\right\rfloor-1 & \text{ if } n \equiv 2 \pmod {7},\\
%\left\lfloor \frac{4n}{7}\right\rfloor  & \text{ otherwise.}
\alpha_n-1 & \text{ if } n \equiv 2 \pmod {7},\\
\alpha_n   & \text{ otherwise.}
\end{cases} 
\]
The result follows.
\end{proof}

\begin{lem}\label{lem:3-0}
%R Suppose that $n \ge 3$ and $n \equiv 0,4,6,7,11,13 \pmod{14}$.
Suppose that $n \ge 3$ and $n \equiv 0,4,6 \pmod{7}$.
If there is an LCD $[n,3,\alpha_n]$ code $C$, then $d(C^\perp) \ge 2$.
\end{lem}
\begin{proof}
%R Write $n=14t+s$, where $0 \le s \le 13$.
%R For $s$ and $\alpha_n$, we have the following:
%R \begin{center}
%R \begin{tabular}{cc|cc|cc|cc|cc}
%R $s$ & $\alpha_n$ &$s$ & $\alpha_n$ &$s$ & $\alpha_n$ &
%R $s$ & $\alpha_n$ &$s$ & $\alpha_n$ \\
%R \hline
%R  0 &$8t$   & 3 &$8t+1$ & 6 &$8t+3$ &  9 &$8t+5$ & 12 &$8t+6$\\
%R  1 &$8t$   & 4 &$8t+2$ & 7 &$8t+4$ & 10 &$8t+5$ & 13 &$8t+7$\\
%R  2 &$8t+1$ & 5 &$8t+2$ & 8 &$8t+4$ & 11 &$8t+6$ & & \\
%R \end{tabular}
%R \end{center}
Write $n=7t+s$, where $0 \le s \le 6$.
For $s$ and $\alpha_n$, we have the following:
\begin{center}
\begin{tabular}{cc|cc|cc|cc}
$s$ & $\alpha_n$ &$s$ & $\alpha_n$ &$s$ & $\alpha_n$ &
$s$ & $\alpha_n$ \\
\hline
 0 &$4t$   & 2 &$4t+1$ & 4 &$4t+2$ & 6 &$4t+3$ \\
 1 &$4t$   & 3 &$4t+1$ & 5 &$4t+2$ &   &\\
\end{tabular}
\end{center}
The result follows by Lemma~\ref{lem:dual2}.
\end{proof}

For nonnegative integers $a,b,c,d,e,f,g,m,\alpha$ and $\delta \in \{0,1\}$, we consider the following
conditions:
\begin{align}
&\alpha \le 1+2(a+b+f+g),   \label{eq:abfg}\\
&\alpha \le 1+2(a+c+e+g)+\delta,   \label{eq:aceg}\\
&\alpha \le 1+2(a+d+e+f)+\delta,   \label{eq:adef}\\
&\alpha \le 2+2(b+c+e+f)+\delta,   \label{eq:bcef}\\
&\alpha \le 2+2(b+d+e+g)+\delta,   \label{eq:bdeg}\\
&\alpha \le 2+2(c+d+f+g),   \label{eq:cdfg}\\
&\alpha \le 3+2(a+b+c+d),   \label{eq:abcd}\\
&a+b+c+d+e+f+g =m.     \label{eq:abcdefg}
%&b \le c \le d.   \label{eq:blecled}
%&e \le f \le g   \label{eq:elefleg}
\end{align}
Define the following sets:
\begin{align*}
R_1=\left\{r \in \ZZ \mid \alpha-m-\frac{3+\delta}{2} \le r \le m-\frac{3}{4}\alpha+\frac{3+\delta}{2} \right\},\\
R_2=\left\{r \in \ZZ \mid \alpha-m-\frac{4+\delta}{2} \le r \le m-\frac{3}{4}\alpha+\frac{2+\delta}{2} \right\}.
\end{align*}

\begin{lem} \label{lem:abcdefg}
Let $a,b,c,d,e,f,g$ be nonnegative integers satisfying the 
conditions~\eqref{eq:abfg}--\eqref{eq:abcdefg}.
%\eqref{eq:ab},
%\eqref{eq:ac},
%\eqref{eq:bc},
%\eqref{eq:abc},
% Then we have
% \begin{align*}
% &a,f,g \in R_1,\\ 
% &b,c,d \in R_2,\\
% &e \in R_1 \text{ if } \delta=0,  e \in R_2 \text{ if } \delta=1,
% \end{align*}
\begin{itemize}
\item[\rm (i)] 
If $\delta=0$, then
$a,e, f,g \in R_1$ and $b,c,d \in R_2$.
\item[\rm (ii)] 
If $\delta=1$, then
$a,f,g \in R_1$ and $b,c,d,e \in R_2$.
\end{itemize}
\end{lem}
\begin{proof}
All cases are similar, and we only give the details
for $a \in R_1$ and $b \in R_2$.

From~\eqref{eq:bcef}, \eqref{eq:bdeg}, \eqref{eq:cdfg} and \eqref{eq:abcdefg}, we have
$a \le m-\frac{3}{4}\alpha+\frac{3+\delta}{2}$.
From~\eqref{eq:abfg}, \eqref{eq:aceg}, \eqref{eq:adef}, \eqref{eq:abcd}
 and \eqref{eq:abcdefg}, we have
$\alpha-m-\frac{3+\delta}{2} \le a$.
Similarly, from~\eqref{eq:aceg}, \eqref{eq:adef}, \eqref{eq:cdfg} and \eqref{eq:abcdefg}, we have
$b \le m-\frac{3}{4}\alpha+\frac{4+\delta}{2}$.
From~\eqref{eq:abfg}, \eqref{eq:bcef}, \eqref{eq:bdeg}, \eqref{eq:abcd}
 and \eqref{eq:abcdefg}, we have
$\alpha-m-\frac{4+\delta}{2} \le b$.
The result follows.
\end{proof}

Now suppose that $C$ and $C'$ are an LCD $[2m+3,3]$ code 
with $d(C^\perp) \ge 2$
and an LCD $[2m+4,3]$ code
with $d(C'^\perp) \ge 2$, respectively for 
$m \ge 1$.
By Propositions~\ref{prop:3cover} and~\ref{prop:3cover2},
we may assume without loss of generality that 
$C$ and $C'$ have generator matrices of the following form:
\begin{align*}
&
\left(
\begin{array}{ccccc}
1&0&0&& \\
0&1&0& M(a,b,c,d,e,f,g) & M(a,b,c,d,e,f,g) \\
0&0&1&&
\end{array}
\right) \text{ and }
\\&
\left(
\begin{array}{cccccc}
1&0&0&&& 0 \\
0&1&0& M(a,b,c,d,e,f,g) & M(a,b,c,d,e,f,g) & 1\\
0&0&1&&& 1
\end{array}
\right),
\end{align*}
respectively,
% where $M(a,b,c,d,e,f,g)$ is a $3 \times m$ matrix of the form:
where 
\begin{equation}\label{eq:Mabcdefg}
M(a,b,c,d,e,f,g)
=
\left(
\begin{array}{ccccccc}
\1_{a} & \1_{b} & \0_{c} & \0_{d} & \0_{e} & \1_{f} & \1_{g} \\
\1_{a} & \0_{b} & \1_{c} & \0_{d} & \1_{e} & \0_{f} & \1_{g} \\
\1_{a} & \0_{b} & \0_{c} & \1_{d} & \1_{e} & \1_{f} & \0_{g}
\end{array}
\right).
\end{equation}
We denote the codes by
$C^0(a,b,c,d,e,f,g)$ and $C^1(a,b,c,d,e,f,g)$, respectively.
Then the codes $C^\delta(a,b,c,d,e,f,g)$ have the following weight enumerators for $\delta \in \{0,1\}$:
\begin{align}\label{eq:3we}
\begin{split}
1
&+y^{1+2(a+b+f+g)}+y^{1+2(a+c+e+g)+\delta}+y^{1+2(a+d+e+f)+\delta}\\
&+y^{2+2(b+c+e+f)+\delta}+y^{2+2(b+d+e+g)+\delta}+y^{2+2(c+d+f+g)}
+y^{3+2(a+b+c+d)}.
\end{split}
\end{align}

\begin{lem}\label{lem:3-2}
There is no LCD $[n,3,\alpha_n]$ code
for $n \equiv 0, 4,6 \pmod {7}$.
\end{lem}
\begin{proof}
There is no LCD $[4,3,2]$ code (see~\cite[Table~1]{bound}).
Assume that $n \equiv 0,4,6 \pmod {7}$ and $n \ge 6$.
Suppose that there is an LCD $[n,3,\alpha_n]$ code $C$.
By Lemma~\ref{lem:3-0}, $d(C^\perp) \ge 2$.
% By Propositions~\ref{prop:3cover},~\ref{prop:3cover2} and
% Lemma~\ref{lem:dual2},
Hence, $C  \cong C^0(a,b,c,d,e,f,g)$ if $n \equiv 7,11,13 \pmod {14}$ 
and $C  \cong C^1(a,b,c,d,e,f,g)$ if $n \equiv 0,4,6 \pmod {14}$
for some $(a,b,c,d,e,f,g)$.

Since $C$ has minimum weight $\alpha_n$, $(a,b,c,d,e,f,g)$ 
satisfies~\eqref{eq:abfg}--\eqref{eq:abcdefg} with $n=3+2m+\delta$ and
$\alpha=\alpha_n$.

\begin{itemize}
\item $(n,\alpha_n)=(14t,8t)$ $(t \ge 1)$: We have
$R_2=\emptyset$,
which is a contradiction. 

\item $(n,\alpha_n)=(14t+4,8t+2)$ $(t \ge 1)$
and
$(14t+6,8t+3)$ $(t \ge 0)$: We have
\[(a,b,c,d,e,f,g)=(t,t,t,t,t,t,t)\]
by Lemma~\ref{lem:abcdefg}.
These contradict~\eqref{eq:abfg} and \eqref{eq:abcdefg},
respectively.

\item $(n,\alpha_n)=(14t+7,8t+4)$ $(t \ge 0)$: We have
$R_1=\emptyset$,
which is a contradiction. 

\item $(n,\alpha_n)=(14t+11,8t+6)$ 
and
$(14t+13,8t+7)$  $(t \ge 0)$: We have
\[(a,b,c,d,e,f,g)=(t+1,t,t,t,t+1,t+1,t+1)\]
by Lemma~\ref{lem:abcdefg}.
These contradict~\eqref{eq:abcd} and \eqref{eq:abcdefg},
respectively.
%
%\item Assume that $n=14t+15$.
%Then $\alpha_n=8t+8$ and we have
%\begin{align*}
%&(a,e,f,g)=(t+1,t+1,t+1,t+1), \\
%&(b,c,d) \in \{(t,t+1,t+1),(t+1,t,t+1),(t+1,t+1,t)\}
%\end{align*}
%by Lemma~\ref{lem:abcdefg}.
%This contradicts the equations~\eqref{eq:abfg},~\eqref{eq:aceg} or
%\eqref{eq:adef}.
\end{itemize}
This completes the proof.
\end{proof}

% \begin{lem}\label{lem:3-3}
% There is no LCD $[n,3,\alpha_n]$ code $C$ with $d(C^\perp) \ge 2$
% for $n \equiv 1,8 \pmod {14}$.
% \end{lem}
% \begin{proof}
% Assume that $n \equiv 1,8 \pmod {14}$.
% Suppose that there is an LCD $[n,3,\alpha_n]$ code $C$  with
% $d(C^\perp) \ge 2$.
% By Propositions~\ref{prop:3cover},~\ref{prop:3cover2}, Lemmas
%~\ref{lem:dual2} and~\ref{lem:3-2},
% $C$ is equivalent to a code $C^0(a,b,c,d,e,f,g)$ if $n \equiv 1 \pmod {14}$,
% and is equivalent to a code $C^1(a,b,c,d,e,f,g)$ if $n \equiv 8 \pmod
% {14}$.
% 
% Since $C$ has minimum weight $\alpha_n$, $(a,b,c,d,e,f,g)$ satisfies the
% conditions~\eqref{eq:abfg}--\eqref{eq:abcdefg} with $n=3+2m+\delta$ and
% $\alpha=\alpha_n$.
% \begin{itemize}
% \item Assume that $n=14t+1$.
% Then $\alpha_n=8t$ and we have
% \[(a,b,c,d,e,f,g)=(t,t,t,t,t,t,t)\]
% by Lemma~\ref{lem:abcdefg}.
% This contradicts the equation~\eqref{eq:abcdefg}.
% 
% \item Assume that $n=14t+8$.
% Then $\alpha_n=8t+4$ and we have
% \begin{align*}
% &(a,f,g) \in \{(t+1,t,t+1),(t+1,t+1,t),(t,t+1,t+1)\}, \\
% &(b,c,d,e)=(t,t,t,t)
% \end{align*}
% by Lemma~\ref{lem:abcdefg}.
% This contradicts the equations~\eqref{eq:bcef},~\eqref{eq:bdeg} or
%~\eqref{eq:abcd}.
% \end{itemize}
% This completes the proof.
% \end{proof}

\begin{lem}\label{lem:3-3}
There is no LCD $[n,3,\alpha_n]$ code for $n \equiv 1 \pmod {7}$.
\end{lem}
\begin{proof}
Assume that $n \equiv 1 \pmod {7}$ and $n \ge 8$.
Suppose that there is an LCD $[n,3,\alpha_n]$ code $C$.
Since $n-1 \equiv 0 \pmod {7}$ and $\alpha_n=\alpha_{n-1}$, we have
\[d(n-1,3) \le \alpha_{n-1}-1=\alpha_{n}-1\]
by Lemma~\ref{lem:3-2}.
By Lemma~\ref{lem:dual2},
$d(C^\perp) \ge 2$.
Hence, %,by Propositions~\ref{prop:3cover} and~\ref{prop:3cover2},
$C  \cong C^0(a,b,c,d,e,f,g)$ if $n \equiv 1 \pmod{14}$ and 
$C  \cong C^1(a,b,c,d,e,f,g)$ if $n \equiv 8 \pmod{14}$
for some $(a,b,c,d,e,f,g)$.

Since $C$ has minimum weight $\alpha_n$, $(a,b,c,d,e,f,g)$ 
satisfies~\eqref{eq:abfg}--\eqref{eq:abcdefg} with $n=3+2m+\delta$ and
$\alpha=\alpha_n$.
\begin{itemize}
\item $(n,\alpha_n)=(14t+1,8t)$ $(t \ge 1)$: We have
\begin{align*}
(a,e,f,g)=(t,t,t,t) \text{ and }  b,c,d \in \{t-1,t\}
\end{align*}
by Lemma~\ref{lem:abcdefg}.
From~\eqref{eq:abcdefg}, $b+c+d=3t-1$.
Hence, we have
\[
(b,c,d) =(t-1,t,t), (t,t-1,t) \text{ and } (t,t,t-1).
\]
These contradict~\eqref{eq:abfg}, \eqref{eq:aceg} and
\eqref{eq:adef}, respectively.

\item $(n,\alpha_n)=(14t+8,8t+4)$  $(t \ge 0)$:  We have
\begin{align*}
a,f,g \in \{t,t+1\} \text{ and }  (b,c,d,e)=(t,t,t,t)
\end{align*}
by Lemma~\ref{lem:abcdefg}.
From~\eqref{eq:abcdefg}, $a+f+g=3t+2$.
Hence, we have
\[
(a,f,g) =(t,t+1,t+1),(t+1,t,t+1) \text{ and }(t+1,t+1,t).
\]
These contradict~\eqref{eq:abcd},
\eqref{eq:bcef} and \eqref{eq:bdeg}, respectively.
\end{itemize}
This completes the proof.
\end{proof}

Hence, from Lemmas~\ref{lem:3-00}, \ref{lem:3-2} and \ref{lem:3-3}, we have
\begin{equation}\label{eq:3new}
d(n,3) \le \alpha_n-1, 
\end{equation}
if $n\equiv 0,1,2,4,6 \pmod7$.

Suppose that $C$ is an LCD $[n,3,\alpha_n]$ code for 
$n \equiv 3,  5 \pmod {7}$ and $n \ge 5$.
By Lemmas~\ref{lem:dual2},
\ref{lem:3-00} and \ref{lem:3-2}, $d(C^\perp) \ge 2$.
Hence, by Propositions~\ref{prop:3cover} and~\ref{prop:3cover2},
$C \cong C^0(a,b,c,d,e,f,g)$ 
if $n \equiv 3,5 \pmod{14}$ and
$C \cong C^1(a,b,c,d,e,f,g)$ 
if $n \equiv 10,12 \pmod{14}$ for some $(a,b,c,d,e,f,g)$.

\begin{lem}\label{lem:3II-0}
\begin{itemize}
\item[\rm (i)]
$C^0(a,b,c,d,e,f,g) \cong  C^0(a,b,d,c,e,g,f)$ \\
$\cong  C^0(a,c,b,d,f,e,g)
\cong C^0(a,c,d,b,f,g,e)  \cong  C^0(a,d,b,c,g,e,f)$ \\
$\cong C^0(a,d,c,b,g,f,e)$.
\item[\rm (ii)]
$C^1(a,b,c,d,e,f,g)  \cong  C^1(a,b,d,c,e,g,f)$.
\end{itemize}
\end{lem}
\begin{proof}
% The straightforward proof is omitted. 
The result follows by considering permutations of rows and
columns of 
the generator matrices of 
$C^0(a,b,c,d,e,f,g)$ and $C^1(a,b,c,d,e,f,g)$.
\end{proof}

By the above lemma, 
we may assume without loss of generality that
\begin{equation}\label{eq:3II-1}
\begin{array}{rrl}
b \le c \le d &\text{ if }& \delta=0,\\
c \le d       &\text{ if }& \delta=1.
\end{array}
\end{equation}

\begin{lem}\label{lem:3II}
Let $S$ be the set of $(a,b,c,d,e,f,g)$
satisfying~\eqref{eq:abfg}--\eqref{eq:abcdefg} and
\eqref{eq:3II-1}.
\begin{itemize}
\item[\rm (i)] 
If $(n,\alpha)=(14t + 3,8t + 1)$ $(t \ge 1)$, then
$S=\{(t, t, t, t, t, t, t)\}$.
\item[\rm (ii)]  
If $(n,\alpha)=(14t + 5,8t + 2)$ $(t \ge 0)$, then
$S=\{(t + 1, t, t, t, t, t, t)\}$.
\item[\rm (iii)] 
If $(n,\alpha)=(14t + 10,8t + 5)$ $(t \ge 0)$, then
$S=\{(t + 1, t, t, t, t, t + 1, t + 1)\}$.
\item[\rm (iv)] 
If $(n,\alpha)=(14t + 12,8t + 6)$ $(t \ge 0)$, then
\[
S=\{(t + 1, t + 1, t, t, t, t + 1, t + 1)\}.
\]
\end{itemize}
\end{lem}
\begin{proof}
All cases are similar, and we only give the details
for (iv), which is the complicated case.

Suppose that $(n,\alpha)=(14t + 12,8t + 6)$ $(t \ge 0)$.
By Lemma~\ref{lem:abcdefg}, 
$R_1=R_2=\{t,t+1\}$.
From~\eqref{eq:abfg},
$4t+\frac{5}{2} \le a+b+f+g$.
Hence, we have
\[
|\{s \in \{a,b,f,g\} \mid s=t+1\}| \ge 3.
\]
From~\eqref{eq:abcdefg},
$a+b+c+d+e+f+g=7t+4$. 
Hence, we have
\[
|\{s \in \{a,b,c,d,e,f,g\} \mid s=t+1\}| =4.
\]
Therefore, we have 
\begin{align*}
(a,b,f,g) \in
\left\{
\begin{array}{l}
(t+1,t+1,t+1,t),(t+1,t+1,t,t+1), \\
(t+1,t,t+1,t+1),(t,t+1,t+1,t+1), \\
(t+1,t+1,t+1,t+1)
\end{array}
\right\}.
\end{align*}
Here, we remark that
\begin{equation}\label{eq:3II}
|\{s \in \{c,d,e\} \mid s=t+1\}|  \le 1. 
\end{equation}

\begin{itemize}
\item $(a,b,f,g)=(t+1,t+1,t+1,t)$:
From~\eqref{eq:aceg}, \eqref{eq:bdeg} and \eqref{eq:cdfg},
we have
\begin{align*}
2t+1 \le c+e,  
2t+\frac{1}{2} \le d+e \text{ and }  
2t+1 \le c+d,
\end{align*}
respectively.
This contradicts~\eqref{eq:3II}.

\item $(a,b,f,g)=(t+1,t+1,t,t+1)$:
From~\eqref{eq:adef}, \eqref{eq:bcef} and \eqref{eq:cdfg},
we have
\begin{align*}
2t+1 \le d+e,  
2t+\frac{1}{2} \le c+e \text{ and } 
2t+1 \le c+d, 
\end{align*}
respectively.
This contradicts~\eqref{eq:3II}.

\item $(a,b,f,g)=(t+1,t,t+1,t+1)$:
From~\eqref{eq:bcef}, \eqref{eq:bdeg} and \eqref{eq:abcd},
we have
\begin{align*}
2t+\frac{1}{2} \le c+e,  
2t+\frac{1}{2} \le d+e \text{ and } 
2t+\frac{1}{2} \le c+d, 
\end{align*}
respectively.
This contradicts~\eqref{eq:3II}.

\item $(a,b,f,g)=(t,t+1,t+1,t+1)$:
From~\eqref{eq:aceg}, \eqref{eq:adef} and \eqref{eq:abcd},
we have
\begin{align*}
2t+1 \le c+e,  
2t+1 \le d+e \text{ and } 
2t+\frac{1}{2} \le c+d, 
\end{align*}
respectively.
This contradicts~\eqref{eq:3II}.
\end{itemize}
The result follows.
\end{proof}

%By Lemma~\ref{lem:3-1}, we have the following:
Therefore, we have the following theorem.

\begin{thm}\label{thm:dim3-1}
For $n \equiv 3,5 \pmod{7}$ and $n \ge 5$, 
there is a unique LCD $[n, 3,\alpha_n]$ code, up to equivalence.
%R \begin{itemize}
%R \item[\rm (i)]
%R For $t \ge 1$,
%R there is a unique LCD $[7t+3, 3,4t + 1]$ code, up to equivalence.
%R \item[\rm (ii)]
%R For $t \ge 0$,
%R there is a unique LCD $[7t+5, 3,4t+2]$ code, up to equivalence.
%R \end{itemize}
\end{thm}

From~\eqref{eq:3new}, we have 
$d(n,3) \le \alpha_n-1$
if $n\equiv 0,1,2,4,6 \pmod7$.
Now we construct an LCD code meeting the bound.
Suppose that $C$ is an LCD $[n,3,\alpha_n-1]$ code for 
$n \equiv 0,  2 \pmod {7}$ and $n \ge 7$.
By Lemmas~\ref{lem:dual2},
\ref{lem:3-00} and \ref{lem:3-2}, $d(C^\perp) \ge 2$.
Hence, by Propositions~\ref{prop:3cover} and~\ref{prop:3cover2},
$C \cong C^0(a,b,c,d,e,f,g)$ 
if $n \equiv 7,9 \pmod{14}$ and
$C \cong C^1(a,b,c,d,e,f,g)$ 
if $n \equiv 0,2 \pmod{14}$ for some $(a,b,c,d,e,f,g)$.
%R Now we consider LCD $[n,3,\alpha_n-1]$ codes for
%R $n \equiv 0,2 \pmod{7}$ and $n \ge 7$.

\begin{lem}\label{lem:3II-2}
Let $S$ be the set of $(a,b,c,d,e,f,g)$
satisfying~\eqref{eq:abfg}--\eqref{eq:abcdefg} and
\eqref{eq:3II-1}.
\begin{itemize}
\item[\rm (i)] 
If $(n,\alpha)=(14t,8t - 1)$ $(t \ge 1)$, then
\[
S =
\{
(t, t - 1, t, t, t - 1, t, t),
(t, t - 1, t - 1, t, t, t, t)
\}.
\]
\item[\rm (ii)] 
If $(n,\alpha)=(14t + 2,8t)$ $(t \ge 1)$, then
\[
S =
\{
(t, t, t, t, t - 1, t, t),
(t, t, t - 1, t, t, t, t)
\}.
\]
\item[\rm (iii)] 
If $(n,\alpha)=(14t + 7,8t + 3)$ $(t \ge 0)$, then
\[
S =
\left\{
\begin{array}{l}
(t, t, t, t, t + 1, t, t + 1),
(t, t, t, t, t + 1, t + 1, t),\\
(t, t, t, t, t, t + 1, t + 1)
\end{array}
\right\}.
\]
\item[\rm (iv)] 
If $(n,\alpha)=(14t + 9,8t + 4)$ $(t \ge 0)$, then
\[
S = 
\left\{
\begin{array}{l}
(t + 1, t, t, t, t + 1, t, t + 1),
(t + 1, t, t, t, t + 1, t + 1, t),\\
(t + 1, t, t, t, t, t + 1, t + 1)
\end{array}
\right\}.
\]
\end{itemize}
\end{lem}
\begin{proof}
All cases are similar, and we only give the details for (i).

Suppose that $(n,\alpha)=(14t,8t -1)$ $(t \ge 1)$.
By Lemma~\ref{lem:abcdefg}, 
$R_1=R_2=\{t-1,t\}$.
From~\eqref{eq:abcdefg},
$a+b+c+d+e+f+g=7t-2$. 
Hence, we have
\begin{equation}\label{eq:3II-2-1}
|\{s \in \{a,b,c,d,e,f,g\} \mid s=t-1\}| =2.
\end{equation}
From~\eqref{eq:abfg}, \eqref{eq:aceg} and \eqref{eq:adef},
we have
\begin{equation}\label{eq:3II-2-2}
\begin{array}{l}
4t-1 \le a+b+f+g, \\
4t-\frac{3}{2} \le a+c+e+g \text{ and } \\
4t-1 \le a+d+e+f,
\end{array}
\end{equation}
respectively.

Now suppose that $a=t-1$.
From~\eqref{eq:3II-2-2}, we have
$
b=c=d=e=f=g=t$.
Since this contradicts~\eqref{eq:3II-2-1}, we have $a=t$.
Suppose that $g=t-1$.
From~\eqref{eq:3II-2-2}, we have
$b=c=d=e=f=t$.
Since this contradicts~\eqref{eq:3II-2-1}, we have $g=t$.
From~\eqref{eq:cdfg}, 
we have
\begin{equation}\label{eq:3II-2-3}
4t-\frac{3}{2} \le c+d+f+g.
\end{equation}
Suppose that $f=t-1$.
From~\eqref{eq:3II-2-2} and \eqref{eq:3II-2-3}, we have
$b=c=d=e=t$.
Since this contradicts~\eqref{eq:3II-2-1}, we have $f=t$.
Suppose that $d=t-1$.
From~\eqref{eq:3II-2-3}, we have
$c=t$,
which contradicts~\eqref{eq:3II-1}.
Therefore, we have
\begin{align*}
(b,c,e) \in \{(t-1,t-1,t), (t-1,t,t-1)\}.
\end{align*}
The result follows.
\end{proof}

We denote the code with generator matrix of the form
$M(a,b,c,d,e,f,g)$ in~\eqref{eq:Mabcdefg}
by $D(a,b,c,d,e,f,g)$.
% \[
% G(a,b,c,d,e,f,g)=
% \left(
% \begin{array}{ccccccc}
% \1_{a} & \1_{b} & \0_{c} & \0_{d} & \0_{e} & \1_{f} & \1_{g} \\
% \1_{a} & \0_{b} & \1_{c} & \0_{d} & \1_{e} & \0_{f} & \1_{g} \\
% \1_{a} & \0_{b} & \0_{c} & \1_{d} & \1_{e} & \1_{f} & \0_{g}
% \end{array}
% \right).
% \]
It is trivial that 
$C^0(a,b,c,d,e,f,g)  \cong D(2a,2b+1,2c+1,2d+1,2e,2f,2g)$ and
$C^1(a,b,c,d,e,f,g)  \cong D(2a,2b+1,2c+1,2d+1,2e+1,2f,2g)$.

\begin{lem}\label{lem:3II-3}
\begin{itemize}
\item[\rm (i)]
For $t \ge 1$,
$D(2t,2t-1,2t+1,2t+1,2t-1,2t,2t) \cong D(2t,2t-1,2t-1,2t+1,2t+1,2t,2t)$.
\item[\rm (ii)]
For $t \ge 1$,
$D(2t,2t+1,2t-1,2t+1,2t+1,2t,2t) \cong
D(2t,2t+1,2t+1,2t+1,2t-1,2t,2t)$.
\end{itemize}
\end{lem}
\begin{proof}
Let $r_i$ be the $i$-th row of $M(a,b,c,d,e,f,g)$.
Consider the following matrices:
\[
\left(
\begin{array}{ccccccc}
r_1\\
r_3\\
r_2+r_3
\end{array}
\right) \text{ and } 
\left(
\begin{array}{ccccccc}
r_1\\
r_2\\
r_2+r_3
\end{array}
\right)
%=
%\left(
%\begin{array}{ccccccc}
%\1_{2t} & \1_{2t-1} & \0_{2t+1} & \0_{2t+1} & \0_{2t-1} & \1_{2t} & \1_{2t} \\
%\1_{2t} & \0_{2t-1} & \0_{2t+1} & \1_{2t+1} & \1_{2t-1} & \1_{2t} & \0_{2t} \\
%\0_{2t} & \0_{2t-1} & \1_{2t+1} & \1_{2t+1} & \0_{2t-1} & \1_{2t} & \1_{2t}
%\end{array}
%\right).
\]
for (i) and (ii), respectively.
The result follows.
\end{proof}

\begin{thm}\label{thm:dim3-2}
For $n \equiv 0,2 \pmod{7}$ and $n \ge 7$, 
there is a unique LCD $[n, 3,\alpha_n-1]$ code, up to equivalence.
%R \begin{itemize}
%R \item[\rm (i)]
%R For $t \ge 1$,
%R there is a unique LCD $[7t,3,4t-1]$ code, up to equivalence.
%R \item[\rm (ii)]
%R For $t \ge 1$,
%R there is a unique LCD $[7t + 2,3,4t]$ code, up to equivalence.
%R \end{itemize}
\end{thm}
\begin{proof}
The result follows from 
Lemmas~\ref{lem:3II-0},
\ref{lem:3II-2} and~\ref{lem:3II-3}.
\end{proof}

\begin{lem}\label{lem:3-4}
There is an LCD $[n,3,\alpha_n-1]$ code 
for $n \equiv 1,4,6 \pmod {7}$ and $n \ge 4$. 
\end{lem}
\begin{proof}
There is an LCD $[4,3,2]$ code (see~\cite[Table~1]{bound}).
Suppose that $n \ge 6$.
Consider the following codes:
\begin{align*}
&C^1(t+1,t,t,t,t,t,t),  \\
&C^1(t,t,t,t,t,t+1,t+1),  \\
&C^0(t+1,t,t,t,t+1,t+1,t+1), \\
&C^0(t+1,t+1,t,t,t+1,t+1,t+1), \\
&C^0(t+1,t+1,t,t+1,t+1,t+1,t+1) \text{ and } \\
&C^1(t+1,t+1,t+1,t+1,t+1,t+1,t+1),
\end{align*}
for $t \ge 0$.
We denote these codes by $C_i$ $(i=1,2,\ldots,6)$, respectively.
The codes $C_i$ have lengths 
$14t+6$, $14t+8$,
$14t+11$, $14t+13$, 
$14t+15$ and $14t+18$, 
respectively.
% By~\eqref{eq:3we}, these codes have the following
% weight enumerators:
The weight enumerators $W_i$ of $C_i$ $(i=1,2,\ldots,6)$
are obtained by~\eqref{eq:3we}, where $W_i$ are listed
in Table~\ref{Tab:W}.
% \[
% \begin{array}{ll}
% 1+y^{8t + 2}+3 y^{8t + 3}+2 y^{8t + 4}+y^{8t + 5},
% &
% 1+3 y^{8t + 3}+2 y^{8t + 4}+y^{8t + 5}+y^{8t + 6},
% \\
% 1+y^{8t + 3}+2 y^{8t + 4}+3 y^{8t + 5}+y^{8t + 6},
% &
% 1+2 y^{8t + 4}+3 y^{8t + 5}+y^{8t + 6}+y^{8t + 7},
% \\
% 1+y^{8t + 5}+3 y^{8t + 6}+3 y^{8t + 7},
% &
% 1+y^{8t + 6}+3 y^{8t + 7}+2 y^{8t + 8}+y^{8t + 9},
% \\
% 1+3 y^{8t + 7}+2 y^{8t + 8}+y^{8t + 9}+y^{8t + 10},
% &
% 1+y^{8t + 7}+2 y^{8t + 8}+3 y^{8t + 9}+y^{8t + 10},
% \\
% 1+2 y^{8t + 8}+3 y^{8t + 9}+y^{8t + 10}+y^{8t + 11},
% &
% 1+y^{8t + 9}+3 y^{8t + 10}+3 y^{8t + 11},
% \end{array}
% \]
% respectively.
The result follows.
% By~\eqref{eq:3we}, these codes have
% minimum weights
% $\alpha_{14t+4}-1=8t+1$,
% $\alpha_{14t+6}-1=8t+2$,
% $\alpha_{14t+7}-1=8t+3$,
% $\alpha_{14t+8}-1=8t+3$,
% $\alpha_{14t+9}-1=8t+4$,
% $\alpha_{14t+11}-1=8t+5$,
% $\alpha_{14t+13}-1=8t+6$
% $\alpha_{14t+14}-1=8t+7$,
% $\alpha_{14t+15}-1=8t+7$
% and
% $\alpha_{14t+16}-1=8t+8$,
% respectively.
\end{proof}

\begin{rem}
For the parameters
$[4,3,1]$,
$[6,3,2]$ and 
$[8,3,3]$,
%R $[11,3,5]$,
%R $[13,3,6]$ and 
%R $[15,3,7]$,
a number of inequivalent LCD codes are known
(see Table~\ref{Tab:C}).
\end{rem}

%%%%%%%%%%%%%%%%%%%%%%%%%%%%%%
\begin{table}[thb]
\caption{$W_{i}$ $(i=1,2,\ldots,6)$}
\label{Tab:W}
\begin{center}
%{\small
{\footnotesize
%{\scriptsize
%{\tiny
\begin{tabular}{c|l|c|l}
\noalign{\hrule height0.8pt}
$i$ & \multicolumn{1}{c|}{$W_{i}$}&$i$ & \multicolumn{1}{c}{$W_{i}$}\\
\hline
 1&$1+y^{8t + 2}+3 y^{8t + 3}+2 y^{8t + 4}+y^{8t + 5}$&
 4&$1+y^{8t + 6}+3 y^{8t + 7}+2 y^{8t + 8}+y^{8t + 9}$\\
 2&$1+y^{8t + 3}+2 y^{8t + 4}+3 y^{8t + 5}+y^{8t + 6}$&
 5&$1+y^{8t + 7}+2 y^{8t + 8}+3 y^{8t + 9}+y^{8t + 10}$\\
 3&$1+y^{8t + 5}+3 y^{8t + 6}+3 y^{8t + 7}$&
 6&$1+y^{8t + 9}+3 y^{8t + 10}+3 y^{8t + 11}$\\
\noalign{\hrule height0.8pt}
\end{tabular}
}
\end{center}
\end{table}
%%%%%%%%%%%%%%%%%%%%%%%%%%%%%%%

Lemmas~\ref{lem:3-00}, \ref{lem:3-2}, \ref{lem:3-3}, \ref{lem:3-4} and
Theorems~\ref{thm:dim3-1}, \ref{thm:dim3-2}
complete the proof of
Theorem~\ref{thm:n3dmax}.

%%%%%%%%%%%%%%%%%%%%%%%%%%%%%%%%%%%%%%%%%%%%%%%%%
\section{Classification of LCD codes for small parameters}
\label{Sec:C}

In this section, we give a complete classification of
LCD $[n,k]$ codes having the minimum weight $d(n,k)$ 
for $2 \le k \le n-1 \le 15$.

We describe how LCD $[n,k]$ codes having the
minimum weight $d(n,k)$ 
were classified.
Let $d_{\text{all}}(n,k)$ denote the largest minimum weight
among all (unrestricted) $[n,k]$ codes.
The values $d_{\text{all}}(n,k)$ can be found in~\cite{G}.
For a fixed pair $(n,k)$, we found all inequivalent
$[n,k]$ codes by one of the following methods.
If there is no LCD $[n,k,d_{\text{all}}(n,k)]$ code, then
we consider the case $d_{\text{all}}(n,k)-1$.

Let $C$ be an $[n,k,d]$ code with parity-check matrix $H$.
Let $D$ be a code with parity-check matrix obtained from $H$
by deleting a column.
The code $D$ is an $[n-1,k-1,d']$ code with $d' \ge d$.
By considering the inverse operation, 
all $[n,k,d]$ codes are obtained from $[n-1,k-1,d']$ codes with $d' \ge d$.
Starting from $[n,1,d']$ codes with $d' \ge d$, 
all $[n+t,1+t,d]$ codes are found for a given $t \ge 1$.
This was done 
by adding
one column at a time, and complete equivalence tests
are carried out for each new column added. 
It is obvious that all codes, which must be checked to achieve
a complete classification, can be obtained.

For some parameters, we employ the following
method, due to the computational complexity.
Every $[n,k,d]$ code is equivalent to a code with generator 
matrix of the form 
$\left(
\begin{array}{cc}
I_k & A
\end{array}
\right)
$,
where $A$ is a $k \times (n-k)$ matrix.
The set of matrices $A$ was constructed, row by row.
%, using a back-tracking algorithm.
Permuting the rows and columns of $A$ gives rise to different 
generator matrices which generate equivalent codes.
Here, we consider a natural (lexicographical) order $<$ on 
the set of the vectors of length $n-k$.
Let $r_i$ be the $i$-th row of $A$.
We consider only matrices $A$, satisfying
the condition $r_1 < r_2 < \cdots < r_k$ and $\wt(r_i) \ge d-1$
if $d \ge 3$ and 
the condition $r_1 \le r_2 \le \cdots \le r_k$ and $\wt(r_i) \ge d-1$
if $d \le 2$.
It is obvious that all codes, which must be checked to achieve
a complete classification, can be obtained.

For $2 \le k \le n-1 \le 15$, 
the numbers $N(n,k,d(n,k))$ of the
inequivalent LCD $[n,k,d(n,k)]$
codes are listed in Table~\ref{Tab:C}, along with 
the values $d(n,k)$.
All generator matrices of the codes in the table
can be obtained electronically from
\url{http://www.math.is.tohoku.ac.jp/~mharada/LCD/}.

%%%%%%%%%%%%%%%%%%%%%%%%%%%%%%
\begin{table}[thb]
\caption{$(d(n,k),N(n,k,d(n,k)))$}
\label{Tab:C}
\begin{center}
%{\small
{\footnotesize
%{\scriptsize
%{\tiny
\begin{tabular}{c|cccccccccccccccccccccc}
\noalign{\hrule height0.8pt}
$n\backslash k$  
& 2& 3& 4& 5& 6& 7& 8\\
\hline
 3&$(2,1)$\\
 4&$(2,2)$&$(1,2)$\\
 5&$(2,3)$&$(2,1)$&$(2,1)$\\
 6&$(3,2)$&$(2,3)$&$(2,4)$&$(1,3)$\\
 7&$(4,1)$&$(3,1)$&$(2,9)$&$(2,2)$&$(2,1)$\\
 8&$(5,1)$&$(3,3)$&$(3,1)$&$(2,9)$&$(2,6)$&$(1,4)$\\
 9&$(6,1)$&$(4,1)$&$(4,1)$&$(3,2)$&$(2,23)$&$(2,3)$&$(2,1)$\\
10&$(6,2)$&$(5,1)$&$(4,5)$&$(3,11)$&$(3,2)$&$(2,23)$&$(2,9)$\\
11&$(6,4)$&$(5,6)$&$(4,20)$&$(4,4)$&$(4,1)$&$(3,1)$&$(2,51)$\\
12&$(7,2)$&$(6,1)$&$(5,6)$&$(4,37)$&$(4,11)$&$(3,22)$&$(2,396)$\\
13&$(8,1)$&$(6,6)$&$(6,2)$&$(5,5)$&$(4,146)$&$(4,4)$&$(3,27)$\\
14&$(9,1)$&$(7,1)$&$(6,16)$&$(5,101)$&$(5,4)$&$(4,301)$&$(4,8)$\\
15&$(10,1)$&$(7,8)$&$(6,89)$&$(6,10)$&$(6,2)$&$(5,1)$&$(4,985)$\\
16&$(10,2)$&$(8,1)$&$(7,7)$&$(6,283)$&$(6,60)$&$(5,1596)$&$(5,1)$\\
\hline
$n\backslash k$  
& 9&10&11&12&13&14&15\\
\hline
10&$(1,5)$\\
11&$(2,4)$&$(2,1)$\\
12&$(2,51)$&$(2,12)$&$(1,6)$\\
13&$(2,619)$&$(2,103)$&$(2,5)$&$(2,1)$\\
14&$(3,31)$&$(2,1370)$&$(2,103)$&$(2,16)$&$(1,7)$\\
15&$(4,2)$&$(3,34)$&$(2,2143)$&$(2,196)$&$(2,7)$&$(2,1)$\\
16&$(4,1772)$&$(4,7)$&$(3,34)$&$(2,4389)$&$(2,196)$&$(2,20)$&$(1,8)$\\
\noalign{\hrule height0.8pt}
\end{tabular}
}
\end{center}
\end{table}
%%%%%%%%%%%%%%%%%%%%%%%%%%%%%%%

% Our classification gives some observations.
% It is conjectured in~\cite{bound} that 
% $d(n,k) \le d(n,k-1)$ for $2 \le k \le n$.
% This is true for even $k$~\cite{bound}.
% % From Table~\ref{Tab:C}, if $n \le 16$ then
% %it is true that $d(n,k) \le d(n,k-1)$ for odd $k$.
% From Table~\ref{Tab:C}, we have the following observation.
% 
% \begin{prop}
% If $n \le 16$, then $d(n,k) \le d(n,k-1)$ for odd $k$.
% \end{prop}

We continue a classification of LCD codes with parameters
$[2m+3,2m,2]$ and $[2m+4,2m+1,2]$.
In Proposition~\ref{prop:3cover}, 
for an LCD $[2m+3,2m,2]$ code $C$, 
there is a  $3$-cover $(Y_1,Y_2,Y_3)$
such that $C^\perp  \cong C((Y_1,Y_2,Y_3))$.
In addition, by Proposition~\ref{prop:disordered}, 
when we consider codes $C(\YY)$ constructed from all $k$-covers $\YY$,
which must be checked to achieve a complete classification, 
it is sufficient to consider
only disordered $k$-covers of unlabelled $m$-sets.
According to~\cite{Clarke},
let $\Tdu(m,k)$ denote
the number of disordered $k$-covers of an unlabelled $m$-set.
The formula $\Tdu(m,k)$ is given in~\cite[Theorem~2]{Clarke}.
For $m \le 7$ and $k \le 8$, $\Tdu(m,k)$
is numerically determined in~\cite[Table~1]{Clarke} 
(see also A005783 in~\cite{OEIS}).
Our computer search shows the following:

\begin{prop}
If $1 \le m \le 11$, then
\[
N(2m+3,2m,2)= 
N(2m+4,2m+1,2)= 
\Tdu(m,3).
\]
\end{prop}

% It is  worthwhile to prove the above equation for all 
% positive integers $m$
% or provide a counter-example for some positive integer $m$.

%%%%%%%%%%%%%%%%%%%%
\bigskip
\noindent
{\bf Acknowledgment.}
This work was supported by JSPS KAKENHI Grant Number 15H03633.
The authors would like to thank Makoto Araya
for his useful discussions.
The authors would also like to thank Yuta Watanabe and
the anonymous referees for helpful comments.

%%%%%%%%%%%%%%%%%%%  References  %%%%%%%%%%%%%%%%%%%%%%%%

%%%%%%%%%%%%%%%%%%%%%%%%%%%%%%%%%

\end{document}